\documentclass[12pt,a4paper]{amsart}
\usepackage[english]{babel}
\usepackage[T1]{fontenc}
\usepackage[latin1]{inputenc}
\usepackage{graphicx}
\usepackage{amsmath,amssymb,amsthm,mathrsfs,,txfonts,ifthen,xr}
\usepackage{soul}
\usepackage{array, multirow}
\usepackage{stmaryrd}
\usepackage{mathrsfs}
\usepackage{color}
\usepackage{hyperref}
\usepackage{enumerate}
\usepackage[all]{xy}
\usepackage{tikz-cd}
\usetikzlibrary{decorations.pathreplacing}

\usepackage{ytableau}

\theoremstyle{plain}
\newtheorem{theo}{Theorem}[section]
\newtheorem{lemma}[theo]{Lemma}
\newtheorem{prop}[theo]{Proposition}
\newtheorem{coro}[theo]{Corollary}

\theoremstyle{definition}

\newtheorem{nota}[theo]{Notation}

\theoremstyle{remark}
\newtheorem{rema}[theo]{Remark}

\def\A{{\rm A}}
\def\B{{\rm B}}
\def\C{{\rm C}}
\def\D{{\rm D}}
\def\E{{\rm E}}
\def\F{{\rm F}}

\def\H{{\rm H}}
\def\I{{\rm I}}

\def\M{{\rm M}}

\def\O{{\rm O}}

\def\p{{\rm p}}
\def\Q{{\rm Q}}
\def\q{{\rm q}}
\def\R{{\rm R}}
\def\S{{\rm S}}
\def\s{{\rm s}}
\def\T{{\rm T}}
\def\U{{\rm U}}
\def\V{{\rm V}}

\def\X{{\rm X}}
\def\Y{{\rm Y}}
\def\Z{{\rm Z}}

\def\Sp{{\rm Sp}}
\def\GL{{\rm GL}}	
\def\End{{\rm End}}
\def\dim{{\rm dim}}

\def\Hom{{\rm Hom}}
\def\Id{{\rm Id}}
\def\Mat{{\rm Mat}}
\def\Hom{{\rm Hom}}
\def\Span{{\rm Span}}
\def\sp{{\rm sp}}
\def\spo{{\rm spo}}
\def\sdim{{\rm sdim}}
\def\str{{\rm str}}

\allowdisplaybreaks

\title{Howe duality for the dual pair $\left(\text{SpO} (2n|1)\,, \mathfrak{osp}(2|2)\right)$}
\author{Roman L\'{a}vi\v{c}ka$^{\mathrm{\lowercase{a}}}$}
\address{Charles University \\ Faculty of Mathematics and Physics \\ Mathematical Institute \\ Sokolovsk\'a 83 \\ 186 75 Praha \\ Czech Republic}
\email{lavicka@karlin.mff.cuni.cz}
\author{Allan Merino$^{\mathrm{\lowercase{b}}}$}
\address{Department of Mathematics and Statistics\\ University of North Florida \\ 1 UNF Drive \\ Jacksonville \\ FL 32224 \\ USA}
\email{allan.merino@unf.edu}
\date{}

\begin{document}

\begin{abstract}

The goal of our work is to study the decomposition of the joint action of $\mathscr{G} = \textbf{SpO}(2n|1)$ and $\mathfrak{g}' = \mathfrak{osp}(2|2)$ 
on the supersymmetric algebra $\S = \S(\mathbb{C}^{2n|1} \otimes \mathbb{C}^{1|1})$. As proved by Merino and Salmasian, we have a one-to-one correspondence between irreducible representations of $\mathscr{G}$ and $\mathfrak{g}'$ appearing as subrepresentations of $\S$. In this paper, we obtained an explicit description of the highest weights and joint highest weight vectors for the representations of $\mathscr{G}$ and $\mathfrak{g}'$ appearing in the duality\,.

\end{abstract}

\maketitle

\tableofcontents

\section{Introduction}

Let $\E = \E_{\bar{0}} \oplus \E_{\bar{1}}$ be a $\mathbb{Z}_{2}$-graded complex vector space, and let $\widetilde{\E} = \E \oplus \E^{*}$. On the vector space $\widetilde{\E}$, we define an even, non-degenerate, $(-1)$-supersymmetric, bilinear form $\widetilde{\B}$, and let $\mathfrak{spo}(\widetilde{\E}) := \mathfrak{spo}(\widetilde{\E}\,, \widetilde{\B})$ be the corresponding orthosymplectic Lie superalgebra\,.

\noindent We denote by $\left(\omega\,, \S\right)$ the spinor-oscillator representation of $\mathfrak{spo}(\widetilde{\E}\,, \widetilde{\B})$. This representation is realized on the supersymmetric algebra $\S := \S(\E)$, and the action of the subalgebra $\mathfrak{gl}(\E) \subseteq \mathfrak{spo}(\widetilde{\E})$ on $\S(\E)$ is the natural extension of $\mathfrak{gl}(\E) \curvearrowright \E$ to $\S(\E)$\,.

\noindent In \cite{ALLANHADI}, the authors defined and classified the irreducible reductive dual pairs in $\mathfrak{spo}(\widetilde{\E})$. Roughly speaking, an irreducible reductive dual pair in $\mathfrak{spo}(\widetilde{\E})$ is a pair of subalgebras $\left(\mathfrak{g}\,, \mathfrak{g}'\right)$ of $\mathfrak{spo}(\widetilde{\E})$ that are centralizer of each other in $\mathfrak{spo}(\widetilde{\E})$, i.e. 
\begin{equation*}
\mathscr{C}_{\mathfrak{spo}(\widetilde{\E})}\left(\mathfrak{g}\right) := \left\{\Y \in \mathfrak{spo}(\widetilde{\E})\,, \left[\Y\,, \X\right] = 0\,, \left(\forall \X \in \mathfrak{g}\right)\right\} = \mathfrak{g}'\,, \qquad \qquad \mathscr{C}_{\mathfrak{spo}(\widetilde{\E})}\left(\mathfrak{g}'\right) = \mathfrak{g}\,,
\end{equation*}
with $\mathfrak{g}\,, \mathfrak{g}'$ acting on $\widetilde{\E}$ reductively, and where the vector space $\widetilde{\E}$ cannot be decomposed as $\widetilde{\E} = \widetilde{\E}_{1} \oplus^{\perp} \widetilde{\E}_{2}$, with $\widetilde{\E}_{1}$ and $\widetilde{\E}_{2}$ $\mathfrak{g} + \mathfrak{g}'$-invariants. As proven in \cite{ALLANHADI}, any irreducible reductive dual pair in the complex orthosymplectic Lie superalgebra $\mathfrak{spo}(\widetilde{\E})$ is isomorphic to one of the following pairs
\begin{enumerate}
\item $\left(\mathfrak{gl}(m|n)\,, \mathfrak{gl}(r\,, s)\right) \subseteq \mathfrak{spo}(2(mr + ns)|2(ms + nr))$
\item $\left(\mathfrak{q}(m)\,, \mathfrak{q}(n)\right) \subseteq \mathfrak{spo}(2mn|2mn)$
\item $\left(\mathfrak{p}(m)\,, \mathfrak{p}(n)\right) \subseteq \mathfrak{spo}(2mn|2mn)$
\item $\left(\mathfrak{spo}(2m|n)\,, \mathfrak{osp}(2r|2s)\right) \subseteq \mathfrak{spo}(2(2mr+ns)|2(2ms + nr))$
\end{enumerate}

Let $\left(\mathfrak{g}\,, \mathfrak{g}'\right)$ be an irreducible reductive dual pair in $\mathfrak{spo}(\widetilde{\E}\,, \widetilde{\B})$. A challenging problem in representation theory is to understand the decomposition of  $\omega_{|_{\mathfrak{g} + \mathfrak{g}'}}$\,.

\begin{itemize}
\item If $\E_{\bar{1}} = \left\{0\right\}$, the representation $\omega$ is the (Fock model of the) metaplectic representation of $\mathfrak{sp}(\widetilde{\E}_{\bar{0}})$ and the restriction of $\omega$ to a reductive dual pair is well-understood (see \cite{HOWE89})\,. 
\item Similarly, if $\E_{\bar{0}} = \left\{0\right\}$, the representation $\omega$ is the spinorial representation of $\mathfrak{so}(\widetilde{\E}_{\bar{1}})$ and the restriction of $\omega$ to a reductive dual pair has been studied in \cite{HOWELECTURENOTES} (see also \cite{ALLANCLEMENTGANG})\,.
\item If $\mathfrak{g}_{\bar{1}} = \left\{0\right\}$ and $\mathfrak{g} \subseteq \mathfrak{gl}(\E)$, i.e. $\left(\mathfrak{g}\,, \mathfrak{g}'\right)$ is one of the pairs
\begin{itemize}
\item $\left(\mathfrak{g}\,, \mathfrak{g}'\right) = \left(\mathfrak{gl}(m)\,, \mathfrak{gl}(r|s)\right)$, with $\E = \mathbb{C}^{m} \otimes \mathbb{C}^{r|s}$
\item $\left(\mathfrak{g}\,, \mathfrak{g}'\right) = \left(\mathfrak{sp}(2m)\,, \mathfrak{osp}(2r|2s)\right)$, with $\E = \mathbb{C}^{2m} \otimes \mathbb{C}^{r|s}$
\item $\left(\mathfrak{g}\,, \mathfrak{g}'\right) = \left(\mathfrak{so}(m)\,, \mathfrak{spo}(2r|2s)\right)$, with $\E = \mathbb{C}^{m} \otimes \mathbb{C}^{r|s}$
\end{itemize}
the restriction of $\omega$ to one of these three dual pairs has been studied in \cite{HOWE89} with an explicit description of the highest weights in \cite{CHENGWANG1}. Note that in this case $\mathfrak{g} \curvearrowright \S(\E)$ by finite dimensional modules (and completely reducible)\,.
\item The restriction of $\omega$ to the pair $\left(\mathfrak{gl}(m|n)\,, \mathfrak{gl}(r|s)\right)$ has been studied in \cite{CHENGWANG2}, and is "equivalent" to the Schur-Sergeev duality. The authors gave an explicit construction of the joint highest weight vectors and these weight vectors will play an important role in our paper. By using similar techniques, an explicit description of $\omega_{|_{\mathfrak{g} + \mathfrak{g}'}}$ has been obtained in \cite{CHENGWANG3} for the pair $\left(\mathfrak{q}(n)\,, \mathfrak{q}(m)\right)$ (see also \cite[Chapter~5]{CHENGWANG1})\,.
\item In \cite{COULEMBIER}, the author looked at the pair $\left(\mathfrak{osp}(m|2n)\,, \mathfrak{sp}(2)\right)$, and obtained a concrete description of $\omega_{|_{\mathfrak{g} + \mathfrak{g}'}}$ in the case when $m - 2n \notin -2\mathbb{Z}^{*}_{+}$. The case when $m - 2n \in -2\mathbb{Z}^{*}_{+}$ has been later on studied in \cite{LAVICKA,LAVICKASMID,SHERMAN}\,.
\end{itemize}

\noindent As mentioned in \cite[Page~548]{HOWE89}, the action of $\mathfrak{gl}(\E)_{\bar{0}} \subseteq \mathfrak{spo}(\widetilde{\E})$ on $\S(\E)$ can be exponentiated to an action of $\GL(\E_{\bar{0}}) \times \GL(\E_{\bar{1}})$ on $\S(\E)$. Using that $\mathfrak{g}_{\bar{0}} := \mathfrak{sp}(2n) \oplus \mathfrak{so}(1) \subseteq \mathfrak{gl}(\E)_{\bar{0}}$, we obtain an action of $\Sp(2n) \times \O(1)$ on $\S(\E)$. We denote by $\mathscr{G} := \textbf{SpO}(2n|1) = \left(\Sp(2n) \times \O(1)\,, \mathfrak{spo}(2n|1)\right)$ the corresponding Harish-Chandra pair (or Lie supergroup, see \cite{CARMELICASTONFIORESI}).

\noindent In \cite{ALLANHADI}, the authors proved that for the pair $\left(\mathscr{G}\,, \mathfrak{g}'\right) = \left(\textbf{SpO}(2n|1)\,, \mathfrak{osp}(2r|2s)\right)$ and $\E = \mathbb{C}^{2n|1} \otimes \mathbb{C}^{r|s}$, we have
\begin{equation}
\S(\E) = \bigoplus\limits_{\pi \in \omega(\mathscr{G})} \V_{\pi} \otimes \V_{\theta(\pi)}
\label{DualityIntroduction}
\end{equation}
where $\omega(\mathscr{G})$ corresponds to the set of (equivalence classes of) irreducible (finite dimensional) representations $\left(\pi\,, \V_{\pi}\right)$ of $\mathscr{G}$ such that $\Hom_{\mathscr{G}}(\pi\,, \omega) \neq \left\{0\right\}$, and where $\V_{\theta(\pi)}$ is an irreducible infinite dimensional representation of $\mathfrak{osp}(2r|2s)$. Moreover, if two irreducible representations $\pi_{1}$ and $\pi_{2}$ in $\omega(\mathscr{G})$ are such that $\pi_{1} \nsim \pi_{2}$, then $\theta(\pi_{1}) \nsim \theta(\pi_{2})$ as $\mathfrak{osp}(2r|2s)$-modules\,.

\noindent In this paper, we give an explicit description of the correspondence $\pi \to \theta(\pi)$ in the case $r = s = 1$. The case $r > 1$ and $s > 1$ is not treated in our paper (see Remark \ref{LastRemark})\,.

\bigskip

Let $\V := \mathbb{C}^{2n|1}$ endowed with an even, non-degenerate, $(-1)$-supersymmetric, bilinear form $\B$, and let $\V' := \mathbb{C}^{2|2}$ endowed with an even, non-degenerate, $1$-supersymmetric, bilinear form $\B'$ (see Section \ref{Section2}). We denote by $\mathfrak{g} := \mathfrak{spo}(\V\,, \B)$ and $\mathfrak{g}' := \mathfrak{osp}(\V'\,, \B')$ the corresponding orthosymplectic Lie superalgebras. Let $\widetilde{\E} = \V \otimes_{\mathbb{C}} \V'$ and let $\widetilde{\B} := \B \otimes \B'$ be the even, non-degenerate, $(-1)$-supersymmetric, bilinear form on $\widetilde{\E}$ defined in Equation \eqref{FormWidetildeB} below. We denote by $\mathfrak{spo}(\widetilde{\E}) := \mathfrak{spo}(\widetilde{\E}\,, \widetilde{\B})$ the corresponding orthosymplectic Lie superalgebra. One can see that the natural actions of $\mathfrak{g}$ and $\mathfrak{g}'$ on $\widetilde{\E}$ preserve the form $\widetilde{\B}$, therefore $\mathfrak{g}$ and $\mathfrak{g}'$ are ($\mathbb{Z}_{2}$-graded) subalgebras of $\mathfrak{spo}(\widetilde{\E})$. Moreover, $\left(\mathfrak{g}\,, \mathfrak{g}'\right)$ form an irreducible reductive dual pair in $\mathfrak{spo}(\widetilde{\E})$\,.

\noindent Since $\V'_{\bar{0}}$ and $\V'_{\bar{1}}$ are both even dimensional, we can find two maximal $\B'$-isotropic subspaces $\X' = \X'_{\bar{0}} \oplus \X'_{\bar{1}}$ and $\Y' = \Y'_{\bar{0}} \oplus \Y'_{\bar{1}}$ of $\V'$ such that $\V' = \X' \oplus^{\perp} \Y'$. In particular, $\X' \cong \mathbb{C}^{1|1}$ and $\Y' \cong \X'^{*}$\,. 

\noindent Let $\E = \V \otimes \X'$ and let $\mathfrak{n}'^{-}\,, \mathfrak{k}'\,,$ and $\mathfrak{n}'^{+}$ be the subalgebras of $\mathfrak{g}'$ respectively given by
\begin{equation*}
\mathfrak{k}' = \left\{\T \in \mathfrak{g}'\,, \T(\X') \subseteq \X'\,, \T(\Y') \subseteq \Y'\right\} \qquad \mathfrak{n}'^{+} = \left\{\T \in \mathfrak{g}'\,, \T(\X') \subseteq \Y'\,, \T(\Y') = \left\{0\right\}\right\}
\end{equation*}
\begin{equation*}
\mathfrak{n}'^{-} = \left\{\T \in \mathfrak{g}'\,, \T(\X') = \left\{0\right\}\,, \T(\Y') \subseteq \X'\right\}
\end{equation*}
One can easily see that $\E$ is $\widetilde{\B}$-isotropic and that  $\left[\mathfrak{k}'\,, \mathfrak{n}'^{+}\right] \subseteq \mathfrak{n}'^{+}$\,.

\noindent We denote by $\S(\E)^{\mathfrak{n}'^{+}}$ the set of supersymmetric tensors annihilated by $\mathfrak{n}'^{+}$, i.e. the subset of $\S(\E)$ given by
\begin{equation*}
\S(\E)^{\mathfrak{n}'^{+}} = \left\{v \in \S(\E)\,, \X \cdot v = 0 \thinspace \thinspace (\forall \X \in \mathfrak{n}'^{+})\right\}
\end{equation*}
The set $\S(\E)^{\mathfrak{n}'^{+}}$ is known as the set of $\mathfrak{g}'$-harmonic supersymmetric tensors on $\E$. Using Equation \eqref{DualityIntroduction}, we get 
\begin{equation*}
\S(\E)^{\mathfrak{n}'^{+}} = \left(\bigoplus\limits_{\pi \in \omega(\mathscr{G})} \V_{\pi} \otimes \V_{\theta(\pi)}\right)^{\mathfrak{n}'^{+}} = \bigoplus\limits_{\pi \in \omega(\mathscr{G})} \V_{\pi} \otimes \V^{\mathfrak{n}'^{+}}_{\theta(\pi)}
\label{EquationDualityIntroduction}
\end{equation*}
and we proved in Section \ref{SectionThree} that $\V^{\mathfrak{n}'^{+}}_{\theta(\pi)}$ is non-zero and $\mathfrak{k}'$-irreducible. In particular, to understand the decomposition of $\S(\E)$ as a $\mathscr{G} + \mathfrak{g}'$-module, it is enough to understand the joint action of $\mathscr{G}$ and $\mathfrak{k}'$ on $\S(\E)^{\mathfrak{n}'^{+}}$. This is analogous to the classical case studied by Kashiwara and Vergne in \cite{KASHIWARAVERGNE}\,.

\noindent In our paper, to obtain \underline{some of} the irreducible representations $\pi \in \omega(\mathscr{G})$ and their corresponding joint $\mathscr{G} + \mathfrak{g}'$-highest weight vectors, we use the results of \cite{CHENGWANG2} and the decomposition of the supersymmetric algebra for the dual pair $\left(\mathfrak{gl}(2n|1)\,, \mathfrak{gl}(1|1)\right)$. Even though direct computations could have been used based on the work of \cite{COULEMBIER}, the use of the $(\mathfrak{gl}\,, \mathfrak{gl})$ duality has two roles:
\begin{enumerate}
\item Generate highest weights and highest weight vectors for $\mathfrak{spo}(2n|1)$ by restriction from $\mathfrak{gl}(2n|1)$\,,
\item Show a major difference with the classical case (i.e. the Howe duality for the oscillator representation of $\mathfrak{sp}(2n)$), where all the representations for the pair $\left(\mathfrak{sp}(2m)\,, \mathfrak{so}(2k)\right)$ are obtained by restriction from the pair $\left(\mathfrak{gl}(2m)\,, \mathfrak{gl}(k)\right)$ (see \cite{HOWELECTURENOTES})\,. 
\end{enumerate}

\noindent To make it easier, we explain our method and results for the pair $\left(\mathfrak{spo}(2|1)\,, \mathfrak{osp}(2|2)\right)$\,. 

\noindent First of all, the pairs $\left(\mathfrak{spo}(2|1)\,, \mathfrak{osp}(2|2)\right)$ and $\left(\mathfrak{gl}(2|1)\,, \mathfrak{gl}(1|1)\right)$ are two irreducible dual pairs in $\mathfrak{spo}(6|6)$ (of type I and II respectively, see \cite{ALLANHADI}). 
\begin{equation*}
\begin{tikzcd}
\mathfrak{gl}(2|1) \arrow[d, dash] \arrow[rd, dash]  & \mathfrak{osp}(2|2) \\
\mathfrak{spo}(2|1) \arrow[ru, dash] & \arrow[u, dash] \mathfrak{gl}(1|1)
\end{tikzcd}
\end{equation*}
These two dual pairs are, in the symplectic case, known as a see-saw dual pair (see \cite{HOWESEESAW}). Moreover, as mentioned previously, the Howe duality and joint highest weights are well-known for the pair $\left(\mathfrak{gl}(2|1)\,, \mathfrak{gl}(1|1)\right)$ (see \cite{CHENGWANG1})\,.

\noindent Let $\mathscr{B} = \left\{u_{1}\,, u_{2}\,, u_{3}\right\}$ be a basis of $\V$ and let $\left\{e\,, f\right\}$ be a basis of $\X'$ such that $\V_{\bar{0}} = \Span\left\{u_{1}\,, u_{2}\right\}\,,$ $\V_{\bar{1}} = \mathbb{C}u_{3}\,, \X'_{\bar{0}} = \mathbb{C}e\,,$ and $\X'_{\bar{1}} = \mathbb{C}f$, and such that 
\begin{equation*}
\Mat_{\mathscr{B}}(\B) = \begin{pmatrix} 0 & 1 & 0 \\ -1 & 0 & 0 \\ 0 & 0 & 1 \end{pmatrix}
\end{equation*}
In particular, we have
\begin{equation*}
\mathfrak{spo}(2|1) = \left\{\begin{pmatrix} a & b & x \\ c & -a & y \\ -y & x & 0 \end{pmatrix}\,, a\,, b\,, c\,, x\,, y \in \mathbb{C}\right\}\,.
\end{equation*}
Let $\mathfrak{t}^{\spo} = \mathbb{C}\left(\E_{1\,, 1} - \E_{2\,, 2}\right)$ be the Cartan subalgebra of $\mathfrak{spo}(2|1)$ and $\mathfrak{b}^{\spo}$ the Borel subalgebra of $\mathfrak{spo}(2|1)$ given by
\begin{equation*}
\mathfrak{b}^{\spo} = \mathfrak{t}^{\spo} \oplus \mathbb{C}\E_{1\,, 2} \oplus \mathbb{C}\left(\E_{1\,, 3} + \E_{3\,, 2}\right)
\end{equation*}
We denote by $x_{1}\,, x_{2}\,, x_{3}\,, \eta_{1}\,, \eta_{2}\,, \eta_{3}$ the basis of $\E$ given by 
\begin{equation*}
x_{1} = u_{1} \otimes e \quad x_{2} = u_{2} \otimes e \quad x_{3} = u_{3} \otimes f \quad \eta_{1} = u_{1} \otimes f \quad \eta_{2} = u_{2} \otimes f \quad \eta_{3} = u_{3} \otimes e\,.
\end{equation*}

\noindent We now use some results of \cite{CHENGWANG2} for the pair $\left(\mathfrak{gl}(2|1)\,, \mathfrak{gl}(1|1)\right)$. We keep the basis of $\E$ defined above. Let $\mathfrak{b}$ be the standard Borel subalgebra of $\mathfrak{gl}(2|1)$, i.e. 
\begin{equation*}
\mathfrak{b} = \mathfrak{t} \oplus \mathfrak{n}^{+} = \left(\bigoplus\limits_{i = 1}^{3}\mathbb{C}\E_{i\,, i}\right) \oplus \left(\mathbb{C}\E_{1\,, 2} \oplus \mathbb{C}\E_{1\,, 3} \oplus \mathbb{C}\E_{2\,, 3}\right)
\end{equation*}
and let $\mathfrak{b}'$ be the standard Borel subalgebra of $\mathfrak{gl}(1|1)$\,.

\noindent In the next proposition, the highest weights for $\mathfrak{gl}(2|1)$ and $\mathfrak{gl}(1|1)$ are given with respect to $\mathfrak{b}$ and $\mathfrak{b}'$ respectively\,.
\begin{prop}

Let $d \in \mathbb{Z}^{+}$.
\begin{enumerate}
\item The vector $x^{d}_{1}$ is a joint highest weight vector for $\mathfrak{gl}(2|1)$ and $\mathfrak{gl}(1|1)$ of highest weights $\left(d\,, 0\,, 0\right)$ and $\left(d\,, 0\right)$ respectively\,.
\item Suppose that $d \geq 2$. The vector $x^{d-2}_{1}\left(x_{1}\eta_{2} - \eta_{1}x_{2}\right)$ is a joint highest weight vector for $\mathfrak{gl}(2|1)$ and $\mathfrak{gl}(1|1)$ of highest weights $\left(d-1\,, 1\,, 0\right)$ and $\left(d-1\,, 1\right)$ respectively\,.
\item Suppose that $d \geq 3$. The vector $x^{k-1}_{1}x^{d-k-2}_{3}\left(x_{1}\eta_{2}x_{3} - x_{2}x_{3}\eta_{1} - (d-k-2)\eta_{1}\eta_{2}\eta_{3}\right)$ is a joint highest weight vector for $\mathfrak{gl}(2|1)$ and $\mathfrak{gl}(1|1)$ of highest weights $\left(k\,, 1\,, d-k-1\right)$ and $\left(k\,, d-k\right)$ respectively\,.
\end{enumerate}

\label{PropositionIntro}

\end{prop}

\noindent One can easily see that $\mathfrak{b}^{\spo}$ is not included in $\mathfrak{b}$, so a $\mathfrak{b}$-highest weight vector is not, in general, a $\mathfrak{b}^{\spo}$-highest weight vector. We denote by $\widetilde{\mathfrak{b}}$ the Borel subalgebra of $\mathfrak{gl}(2|1)$ given by
\begin{equation*}
\widetilde{\mathfrak{b}} = \left\{\begin{pmatrix} \star & \star & \star \\ 0 & \star & 0 \\ 0 & \star & \star \end{pmatrix}\right\}
\end{equation*}
We have $\dim(\widetilde{\mathfrak{b}}) = 6$ and $\mathfrak{b}^{\spo} \subseteq \widetilde{\mathfrak{b}}$. The following lemma transforms a $\mathfrak{b}$-highest weight module of $\mathfrak{gl}(2|1)$ into a $\widetilde{\mathfrak{b}}$-highest weight module\,.

\begin{lemma}

Let $\left(\pi\,, \V_{\pi}\right)$ be an irreducible $\mathfrak{b}$-highest weight module of highest weight vector $v$ and highest weight $\lambda = \left(\lambda_{1}\,, \lambda_{2}\,, \lambda_{3}\right)$\,.
\begin{itemize}
\item If $\lambda\left(\E_{2\,, 2} + \E_{3\,, 3}\right) = 0$, then $\pi$ is a $\widetilde{\mathfrak{b}}$-highest weight module of highest weight vector $v$ and highest weight $\lambda$\,.
\item If $\lambda\left(\E_{2\,, 2} + \E_{3\,, 3}\right) \neq 0$, then $\pi$ is a $\widetilde{\mathfrak{b}}$-highest weight module of highest weight vector $\E_{3\,, 2}v$ and highest weight $\left(\lambda_{1}\,, \lambda_{2}-1\,, \lambda_{3}+1\right)$\,.
\end{itemize}

\label{LemmaIntro}

\end{lemma}

\noindent In particular, by combining Proposition \ref{PropositionIntro} and Lemma \ref{LemmaIntro}, we obtain the following proposition, where the highest weights are given with respect to $\widetilde{\mathfrak{b}}$ and $\mathfrak{b}'$ respectively\,.

\begin{prop}

Let $d \in \mathbb{Z}^{+}$\,.
\begin{enumerate}
\item The vector $x^{d}_{1}$ is a joint highest weight vector for $\mathfrak{gl}(2|1)$ and $\mathfrak{gl}(1|1)$ of highest weights $\left(d\,, 0\,, 0\right)$ and $\left(d\,, 0\right)$ respectively\,.
\item Suppose that $d \geq 2$. The vector $x^{d-2}_{1}\left(x_{1}x_{3} + \eta_{1}\eta_{3}\right)$ is a joint highest weight vector for $\mathfrak{gl}(2|1)$ and $\mathfrak{gl}(1|1)$ of highest weights $\left(d-1\,, 0\,, 1\right)$ and $\left(d-1\,, 1\right)$ respectively\,.
\item Suppose that $d \geq 3$. For $1 \leq k \leq d-2$, the vector $x^{k-1}_{1}x^{d-k-1}_{3}\left(x_{1}x_{3} + (d-k-1)\eta_{1}\eta_{3}\right)$ is a joint highest weight vector for $\mathfrak{gl}(2|1)$ and $\mathfrak{gl}(1|1)$ of highest weights $\left(k\,, 0\,, d-k\right)$ and $\left(k\,, d-k\right)$ respectively\,.
\end{enumerate}

\label{PropositionIntro2}

\end{prop}

\noindent In particular, using Proposition \ref{PropositionIntro2}, we get a family of joint highest weight vectors for $\mathfrak{gl}(2|1)$ and $\mathfrak{gl}(1|1)$ (with respect to $\widetilde{\mathfrak{b}}$ and $\mathfrak{b}'$ respectively). Therefore, by restricting to $\mathfrak{spo}(2|1)$, we get a family of joint $\mathfrak{b}^{\spo} + \mathfrak{b}'$-highest weight vectors for $\mathfrak{spo}(2|1)$ and $\mathfrak{gl}(1|1)$. However, as explain in Equation \eqref{EquationDualityIntroduction}, we need to understand the decomposition of $\S(\E)^{\mathfrak{n}'^{+}}$ as a $\mathfrak{spo}(2|1) + \mathfrak{gl}(1|1)$-module. We first prove the following lemma\,.

\begin{lemma}

Let $d \geq 2$. Among the highest weight vectors obtained in Proposition \ref{PropositionIntro2}, the only vectors that belong to $\S^{d}(\E)^{\mathfrak{n}^{+}}$ are
\begin{equation*}
x^{d}_{1} \qquad \qquad \qquad x^{d-2}_{1}\left(x_{1}x_{3} + \eta_{1}\eta_{3}\right)
\end{equation*}
and the corresponding joint highest weights are respectively given by $\left(d\right) \times \left(d + \frac{1}{2}\,, -\frac{1}{2}\right)$ and $\left(d-1\right) \times \left(d-\frac{1}{2}\,, \frac{1}{2}\right)$\,.

\label{LemmaIntro2}

\end{lemma}

\noindent However, in general, we do not obtain all the representations $\pi \in \omega(\mathscr{G})$ with this method. In the case of $\mathfrak{spo}(2|1)$, we are missing the trivial representation of $\mathfrak{spo}(2|1)$ in $\S^{3}(\E)$. Indeed, let $\widetilde{v}$ be the element in $\S^{3}(\E)$ given by $\widetilde{v} = x_{1}x_{3}\eta_{2} - x_{2}x_{3}\eta_{1} - x^{2}_{3}\eta_{3} - \eta_{1}\eta_{2}\eta_{3}$. Using the explicit actions of $\mathfrak{spo}(2|1)$ and $\mathfrak{gl}(1|1)$ given in Sections \ref{SectionFour} and \ref{SectionGLGL}, we obtain that $\widetilde{v}$ is in $ \S^{3}(\E)^{\mathfrak{n}'^{+}}$ and that $\widetilde{v}$ is a joint $\mathfrak{spo}(2|1)+\mathfrak{gl}(1|1)$-highest weight vector of highest weight $\left(0\right)$ and $\left(\frac{3}{2}\,, \frac{3}{2}\right)$\,. Finally, we obtain the following theorem\,.
 
\begin{theo}

Let $d \geq 4$. As a $\mathfrak{spo}(2|1) + \mathfrak{gl}(1|1)$-module, we have
\begin{equation*}
\S^{d}(\E)^{\mathfrak{n}'^{+}} = \V^{2|1}_{(d)} \otimes \V^{1|1}_{(d+ \frac{1}{2}\,, -\frac{1}{2})} \oplus \V^{2|1}_{(d-1)} \otimes \V^{1|1}_{(d - \frac{1}{2}\,, \frac{1}{2})}
\end{equation*}
where $\V^{2|1}_{(\alpha)}$ (resp. $\V^{1|1}_{(\beta\,, \gamma)}$) corresponds to the $\mathfrak{b}^{\spo}$-highest weight module of $\mathfrak{spo}(2|1)$ (resp. $\mathfrak{b}'$-highest weight module of $\mathfrak{gl}(1|1)$) of highest weight $\left(\alpha\right)$ (resp. $\left(\beta\,, \gamma\right)$). Moreover, we have 
\begin{itemize}
\item $\S^{0}(\E)^{\mathfrak{n}'^{+}} = \V^{2|1}_{(0)} \otimes \V^{1|1}_{(\frac{1}{2}\,, -\frac{1}{2})}$
\item $\S^{1}(\E)^{\mathfrak{n}'^{+}} = \V^{2|1}_{(1)} \otimes \V^{1|1}_{(\frac{3}{2}\,, -\frac{1}{2})}$
\item $\S^{2}(\E)^{\mathfrak{n}'^{+}} = \V^{2|1}_{(2)} \otimes \V^{1|1}_{(\frac{5}{2}\,, -\frac{1}{2})} \oplus \V^{2|1}_{(1)} \otimes \V^{1|1}_{(\frac{3}{2}\,, \frac{1}{2})}$
\item $\S^{3}(\E)^{\mathfrak{n}'^{+}} = \V^{2|1}_{(3)} \otimes \V^{1|1}_{(\frac{7}{2}\,, -\frac{1}{2})} \oplus \V^{2|1}_{(2)} \otimes \V^{1|1}_{(\frac{5}{2}\,, \frac{1}{2})} \oplus \V^{2|1}_{(0)} \otimes \V^{1|1}_{(\frac{3}{2}\,, \frac{3}{2})}$
\end{itemize}
\label{TheoremIntro}

\end{theo}

\begin{rema}
We finish with a few remarks concerning the decomposition obtained in Theorem \ref{TheoremIntro}\,.
\begin{enumerate}
\item We denote by $\S^{+}(\E)$ and $\S^{-}(\E)$ the subsets of $\S(\E)$ defined as
\begin{equation*}
\S^{+}(\E) = \bigoplus\limits_{d = 0}^{\infty} \S^{2d}(\E)\,, \qquad \qquad \S^{-}(\E) = \bigoplus\limits_{d = 0}^{\infty} \S^{2d+1}(\E)\,.
\end{equation*}
By combining Equation \eqref{EquationDualityIntroduction} and Theorem \ref{TheoremIntro}, we get that, as a $\textbf{SpO}(2|1) + \mathfrak{osp}(2|2)$-module,
\begin{equation*}
\S^{+}(\E) = \bigoplus\limits_{d = 0}^{\infty} \V^{2|1}_{(d)} \otimes \V^{2|2}_{\left(d+\frac{1}{2}\,,\frac{(-1)^{d+1}}{2}\right)}
\end{equation*}
Similarly, we have
\begin{equation*}
\S^{-}(\E) = \V^{2|1}_{(0)} \otimes \V^{1|1}_{(\frac{3}{2}\,, \frac{3}{2})} \oplus \bigoplus\limits_{d = 1}^{\infty} \V^{2|1}_{(d)} \otimes \V^{2|2}_{\left(d+\frac{1}{2}\,, \frac{(-1)^{d}}{2}\right)}
\end{equation*}
\item Let $d \geq 1$. One can see that the $\mathfrak{b}^{\spo}$-highest weight representation $\V^{2|1}_{\left(d\right)}$ of $\mathfrak{spo}(2|1)$ of highest weight $\left(d\right)$ appears in $\S^{+}(\E)$ and $\S^{-}(\E)$. However, the corresponding representations of $\mathfrak{osp}(2|2)$ are not isomorphic. This is something that Coulembier obtained for the pair $\left(\mathfrak{osp}(1|2n)\,, \mathfrak{sp}(2)\right)$ (see \cite{COULEMBIER})\,.

\noindent One way to distinguish the representations $\V^{2|1}_{(d)}$ of $\mathfrak{spo}(2|1)$ appearing in $\S^{+}(\E)$ and $\S^{-}(\E)$ is to look at the action of the corresponding Lie supergroup $\textbf{SpO}(2|1)$. For all $\left(\pi\,, \V_{\pi}\right) \in \omega(\mathscr{G})$, we denote by $\widetilde{\pi}$ the corresponding representation of $\textbf{SpO}(2|1)$ on $\V_{\pi}$. Using that $\O(1) = \left\{\pm 1\right\}$, we obtain that $\left\{\pm \Id\right\} \in \Sp(2) \times \O(1)$\,.

\noindent One can easily see that for every homogeneous $v \in \S(\E)$, we have
\begin{equation*}
\widetilde{\pi}(\Id)v = v\,, \qquad \quad \widetilde{\pi}(-\Id)v = (-1)^{\left|v\right|}v\,, \qquad \qquad \quad \left(\pi \in \omega(\mathscr{G})\right)\,.
\end{equation*}
Let $d \geq 1$ and let $\left(\V^{1}_{d}\,, \pi_{1}\right)$ and $\left(\V^{2}_{d}\,, \pi_{2}\right)$ be the representations of $\mathfrak{spo}(2|1)$ isomorphic to $\V^{2|1}_{(d)}$ that are subrepresentations of $\S^{+}(\E)$ and $\S^{-}(\E)$ respectively. We denote by $\widetilde{\pi}_{1}$ and $\widetilde{\pi}_{2}$ the corresponding representations of $\textbf{SpO}(2|1)$. One can see that the representations $\widetilde{\pi}_{1}$ and $\widetilde{\pi}_{2}$ are not equivalent as $\textbf{SpO}(2|1)$-modules. Indeed, suppose that there exists an invertible map $\T \in \Hom\left(\V^{1}_{d}\,, \V^{2}_{d}\right)$ such that 
\begin{equation*}
\T \circ \widetilde{\pi}_{1}(g) = \widetilde{\pi}_{2}(g) \circ \T\,, \qquad \qquad \left(g \in \Sp(2) \times \O(1)\right)\,.
\end{equation*}
Then for all $v \in \V^{1}_{d} \subseteq \S^{+}(\E)$, we have
\begin{equation*}
\left(\T \circ \widetilde{\pi}_{1}(-\Id)\right)(v) = \T\left(\widetilde{\pi}_{1}(-\Id)v\right) = \T(v)
\end{equation*}
and using that $\T(v) \in \V^{2}_{d} \in \S^{-}(\E)$, we get
\begin{equation*}
\left(\widetilde{\pi}_{2}(-\Id) \circ \T\right)(v) = \widetilde{\pi}_{2}(-\Id)\left(\T(v)\right) = -\T(v)
\end{equation*}
i.e. $\T(v) = -\T(v)$ for all $v \in \V^{1}_{d}$, which is impossible. Therefore, as $\textbf{SpO}(2|1)$-modules, $\widetilde{\pi}_{1} \nsim \widetilde{\pi}_{2}$\,.
\end{enumerate}

\label{LastRemarkIntroduction}

\end{rema}

\bigskip

\noindent \textbf{\underline{Acknowledgements: }} The first author acknowledges support of the grant GACR 24-10887S. 
The second author acknowledges support of the Institut Henri Poincar\'e (UAR 839 CNRS-Sorbonne Universit\'e), and LabEx CARMIN (ANR-10-LABX-59-01) where this paper was finished\,.

\section{The dual pair \texorpdfstring{$\left(\mathfrak{g}\,, \mathfrak{g}'\right) = \left(\mathfrak{spo}(2n|1)\,, \mathfrak{osp}(2k|2l)\right)$}{(spo(2n|1), osp(2k|2l))}}

\label{Section2}

Let $\V = \V_{\bar{0}} \oplus \V_{\bar{1}}$ be a $\mathbb{Z}_{2}$-graded complex vector space endowed with an even, non-degenerate, $(-1)$-supersymmetric, bilinear form $\B$, i.e.
\begin{equation*}
\B(\V_{\bar{0}}\,, \V_{\bar{1}}) = \left\{0\right\}\,, \qquad \B(v_{1}\,, v_{2}) = -(-1)^{\left|v_{1}\right|\left|v_{2}\right|}\B(v_{2}\,, v_{1})\,, \qquad \left(v_{1}\,, v_{2} \in \V\right)\,.
\end{equation*} 
We denote by $\mathfrak{g} := \mathfrak{spo}(\V\,, \B)$ the corresponding orthosymplectic Lie superalgebra, i.e. the subalgebra of $\mathfrak{gl}(\V)$ given by 
\begin{equation*}
\mathfrak{spo}(\V\,, \B) = \left\{\X \in \mathfrak{gl}(\V)\,, \B\left(\X(v_{1})\,, v_{2}\right) + (-1)^{\left|\X\right|\left|v_{1}\right|}\B\left(v_{1}\,, \X(v_{2})\right) = 0\,, v_{1}\,, v_{2} \in \V\right\}\,.
\end{equation*}
Similarly, let $\V' = \V'_{\bar{0}} \oplus \V'_{\bar{1}}$ be a $\mathbb{Z}_{2}$-graded complex vector space endowed with an even, non-degenerate, $1$-supersymmetric, bilinear form $\B'$, i.e.
\begin{equation*}
\B'(\V'_{\bar{0}}\,, \V'_{\bar{1}}) = \left\{0\right\}\,, \qquad \B'(v'_{1}\,, v'_{2}) = (-1)^{\left|v'_{1}\right|\left|v'_{2}\right|}\B'(v'_{2}\,, v'_{1})\,, \qquad \left(v'_{1}\,, v'_{2} \in \V'\right)\,.
\end{equation*} 
We denote by $\mathfrak{g}' := \mathfrak{osp}(\V'\,, \B')$ the corresponding orthosymplectic Lie superalgebra, i.e.
\begin{equation*}
\mathfrak{osp}(\V'\,, \B') = \left\{\X' \in \mathfrak{gl}(\V')\,, \B'\left(\X'(v'_{1})\,, v'_{2}\right) + (-1)^{\left|\X'\right|\left|v'_{1}\right|}\B'\left(v'_{1}\,, \X'(v'_{2})\right) = 0\,, v'_{1}\,, v'_{2} \in \V'\right\}\,.
\end{equation*}
On the set $\widetilde{\E} = \V \otimes_{\mathbb{C}} \V'$, we define the form $\widetilde{\B}$ by
\begin{equation}
\widetilde{\B}\left(v_{1} \otimes v'_{1}\,, v_{2} \otimes v'_{2}\right) = (-1)^{\left|v'_{1}\right|\left|v_{2}\right|} \B\left(v_{1}\,, v_{2}\right)\B'\left(v'_{1}\,, v'_{2}\right)\,,
\label{FormWidetildeB}
\end{equation}
and one can easily see that the bilinear form $\widetilde{\B}$ is even, non-degenerate, and $(-1)$-supersymmetric. Let $\mathfrak{s} := \mathfrak{spo}(\widetilde{\E}\,, \widetilde{\B})$ be the corresponding orthosymplectic Lie superalgebra. The Lie superalgebras $\mathfrak{g}$ and $\mathfrak{g}'$ act on $\widetilde{\E}$ by
\begin{equation*}
\X \cdot v \otimes v' = \X(v) \otimes v'\,, \qquad \X' \cdot v \otimes v' = (-1)^{\left|\X'\right|\left|v\right|}v \otimes \X'(v')\,, \qquad \left(\X \in \mathfrak{g}\,, \X' \in \mathfrak{g}'\,, v \in \V\,, v' \in \V'\right)\,,
\end{equation*}
and one can see easily that the form $\widetilde{\B}$ is $\mathfrak{g}+\mathfrak{g}'$-invariant. In particular, $\mathfrak{g}$ and $\mathfrak{g}'$ are seen as subalgebras of $\mathfrak{s}$ and as proved in \cite{NISHIYAMA} (see also \cite{ALLANHADI}), we have
\begin{equation*}
\mathscr{C}_{\mathfrak{s}}\left(\mathfrak{g}\right) := \left\{\Y \in \mathfrak{s}\,, \left[\X\,, \Y\right] = 0\,, \left(\forall \X \in \mathfrak{g}\right)\right\} = \mathfrak{g}'\,, \qquad \qquad \mathscr{C}_{\mathfrak{s}}\left(\mathfrak{g}'\right) = \mathfrak{g}\,,
\end{equation*}
i.e. $\left(\mathfrak{g}\,, \mathfrak{g}'\right)$ is a dual pair in $\mathfrak{s}$. Moreover, the dual pair $\left(\mathfrak{g}\,, \mathfrak{g}'\right)$ is reductive, irreducible, of type I (see \cite[Section~2]{ALLANHADI})\,.

\noindent From now on, we assume that the superdimension $\sdim(\V')$ of $\V'$ is even, where $\sdim(\V') = \dim(\V'_{\bar{0}}) - \dim(\V'_{\bar{1}}$) . Let $\X' = \X'_{\bar{0}} \oplus \X'_{\bar{1}}$ be a $\mathbb{Z}_{2}$-graded subspace of $\V'$, where $\X'_{\alpha}$ is a maximal $\B'$ isotropic subspace of $\V'_{\alpha}$ ($\alpha \in \mathbb{Z}_{2}$). Using that $\sdim(\V')$ is even, it follows that $\V' \cong \X' \oplus \X'^{*}$. Let $\E = \V \otimes_{\mathbb{C}} \X'$. Using that $\X'$ is a maximal $\B'$-isotropic subspace in $\V'$, it follows that $\E$ is a maximal $\widetilde{\B}$-isotropic subspace in $\widetilde{\E}$, and $\widetilde{\E} \cong \E \oplus \E^{*}$\,.

\noindent We denote by $\S(\E)$ the supersymmetric algebra on $\E$. For all $e \in \E$ and $e^{*} \in \E^{*}$, we denote by $\M_{e}$ and $\D_{e^{*}}$ the multiplication and derivation operators on $\S(\E)$, and let $\iota: \E \oplus \E^{*} \to \End(\S(\E))$ be the embedding of $\widetilde{\E}$ into $\End(\S(\E))$ given by 
\begin{equation}
\iota\left(e + e^{*}\right) = \M_{e} + \D_{e^{*}}\,, \qquad \qquad \left(e \in \E\,, \widetilde{e} \in \widetilde{\E}\right)\,.
\label{MapIota}
\end{equation}
We denote by $\textbf{WC}(\widetilde{\E}\,, \widetilde{\B})$ the subalgebra generated by the operators $\iota(\widetilde{e})\,, \widetilde{e} \in \widetilde{\E}$, known as the Weyl-Clifford algebra (see \cite[Section~2]{HOWE89} or \cite[Section~5.1]{CHENGWANG1})\,.

\noindent Let $\Omega$ be the subspace of $\textbf{WC}(\widetilde{\E}\,, \widetilde{\B})$ given by 
\begin{equation}
\Omega = \Omega_{\bar{0}} \oplus \Omega_{\bar{1}} = \S^{2}(\iota(\widetilde{\E}_{\bar{0}})) \oplus \Lambda^{2}(\iota(\widetilde{\E}_{\bar{1}})) \oplus \iota(\widetilde{\E}_{\bar{0}}) \otimes \iota(\widetilde{\E}_{\bar{1}})\,.
\label{SuperalgebraOmega}
\end{equation}
As explained in \cite{HOWE89}, we have $\Omega \cong \mathfrak{spo}(\widetilde{\E}\,, \widetilde{\B})$, and using that
\begin{equation*}
\mathfrak{spo}(\widetilde{\E}\,, \widetilde{\B}) \cong \Omega \subseteq \textbf{WC}(\widetilde{\E}\,, \widetilde{\B}) \subseteq \End(\S(\E)) \curvearrowright \S(\E)\,,
\end{equation*}
we get an action of $\mathfrak{spo}(\widetilde{\E}\,, \widetilde{\B})$ on the supersymmetric algebra $\S(\E)$, known as the spinor-oscillator representation of the orthosymplectic Lie superalgebra $\mathfrak{spo}(\widetilde{\E}\,, \widetilde{\B})$\,.

\begin{nota}

We denote by $\omega$ (or $\left(\omega\,, \S(\E)\right)$) the representation of $\mathfrak{spo}(\widetilde{\E}\,, \widetilde{\B})$ on $\S(\E)$\,.

\end{nota}

\begin{rema}

\begin{enumerate}
\item We have $\E = \V \otimes_{\mathbb{C}} \X'$. In particular, the Lie superalgebra $\mathfrak{g}$ acts on $\S^{d}(\E)$ for all $d \in \mathbb{Z}^{+}$. In other words, $\mathfrak{g}$ acts on $\S(\E)$ by finite dimensional representations. However, it does not imply that the $\mathfrak{g}$-action is completely reducible\,.
\item As mentioned in the introduction, the action of $\mathfrak{gl}(\E)_{\bar{0}}$ can be exponentiated to an action of the supergroup $\textbf{GL}(\E) := \left(\GL(\E_{\bar{0}}) \times \GL(\E_{\bar{1}})\,, \mathfrak{gl}(\E)\right)$ on $\S(\E)$. Let $\mathscr{G} := \textbf{SpO}(\E) \subseteq \textbf{GL}(\E)$ the orthosymplectic Lie supergroup corresponding to $\mathfrak{g}$. Using that $\mathfrak{g}_{\bar{0}} \subseteq \mathfrak{gl}(\E)_{\bar{0}}$, it follows that $\mathscr{G}$ acts on $\S(\E)$\,.
\end{enumerate}

\end{rema}

\noindent As explained in \cite[Section~9]{ALLANHADI}, if $\mathscr{G} := \textbf{SpO}(2n|1)$, then $\omega_{|_{\mathscr{G}}}$ is a direct sum of finite-dimensional irreducible representations of $\mathscr{G}$. More precisely, for every irreducible representation $\left(\pi\,, \V_{\pi}\right)$ of $\mathscr{G}$, we denote by $\V(\pi)$ the corresponding $\pi$-isotypic component in $\S(\E)$, i.e. 
\begin{equation*}
\V(\pi) = \left\{\T(\V_{\pi})\,, \T \in \Hom_{\mathscr{G}}(\V_{\pi}\,, \S(\E))\right\}\,.
\end{equation*}
Therefore, we get
\begin{equation*}
\S(\E) = \bigoplus\limits_{\pi \in \omega(\mathscr{G})} \V(\pi)\,,
\end{equation*}
where $\omega(\mathscr{G})$ is the set of equivalence classes of finite dimensional irreducible representations $\left(\pi\,, \V_{\pi}\right)$ of $\mathscr{G}$ such that $\Hom_{\mathscr{G}}\left(\V_{\pi}\,, \S(\E)\right) \neq \left\{0\right\}$\,.

\noindent We denote by $\V_{\theta(\pi)}$ the set $\V_{\theta(\pi)} := \Hom_{\mathscr{G}}(\V_{\pi}\,, \S(\E))$, i.e.
\begin{equation*}
\V(\pi) = \V_{\pi} \otimes \V_{\theta(\pi)}\,.
\end{equation*}
Let $\textbf{WC}(\widetilde{\E}\,, \widetilde{\B})^{\mathscr{G}}$ be the set of $\mathscr{G}$-invariants in the Weyl-Clifford algebra $\textbf{WC}(\widetilde{\E}\,, \widetilde{\B})$, i.e.
\begin{equation*}
\textbf{WC}(\widetilde{\E}\,, \widetilde{\B})^{\mathscr{G}} = \left\{\Q \in \textbf{WC}(\widetilde{\E}\,, \widetilde{\B})\,, \omega(\T) \circ \Q = (-1)^{\left|\T\right| \cdot \left|\X\right|}\Q \circ \omega(\T)\,, \left(\forall \T \in \mathscr{G}\right)\right\}\,.
\end{equation*}
The Lie superalgebra $\textbf{WC}(\widetilde{\E}\,, \widetilde{\B})^{\mathscr{G}}$ acts naturally on the vector space $\V_{\theta(\pi)}$ by
\begin{equation*}
\Y \cdot \T = (-1)^{\left|\Y\right|\left|\T\right|}\T \circ \Y\,, \qquad \qquad \left(\Y \in \textbf{WC}(\widetilde{\E}\,, \widetilde{\B})^{\mathscr{G}}\,, \T \in \V_{\theta(\pi)}\right)\,.
\end{equation*}

\begin{theo}{\cite[Sections~8~and~9]{ALLANHADI}}

\begin{enumerate}
\item The $\textbf{WC}(\widetilde{\E}\,, \widetilde{\B})^{\mathscr{G}}$-module $\V_{\theta(\pi)}$ is irreducible\,.
\item (Double commutant theorem) The superalgebra $\textbf{WC}(\widetilde{\E}\,, \widetilde{\B})^{\mathscr{G}}$ is generated by $\mathfrak{g}'$\,.
\item The action of $\mathfrak{g}'$ on $\V_{\theta(\pi)}$ is irreducible. Moreover, if $\pi_{1}$ and $\pi_{2}$ are two elements in $\omega(\mathscr{G})$ such that $\pi_{1} \nsim \pi_{2}$, then $\V_{\theta(\pi_{1})} \nsim \V_{\theta(\pi_{2})}$ (as $\mathfrak{g}'$-modules)\,.
\end{enumerate}

\label{TheoremDualityAllanHadi}

\end{theo}

\noindent Therefore, as a $\mathscr{G} + \mathfrak{g}'$-module, we have
\begin{equation}
\S(\E) = \bigoplus\limits_{\pi \in \omega(\mathscr{G})} \V_{\pi} \otimes \V_{\theta(\pi)}\,.
\label{DecompositionDuality}
\end{equation}

\begin{rema}

As explained in \cite[Section~8]{ALLANHADI}, the subalgebra of $\textbf{WC}(\widetilde{\E}\,, \widetilde{\B})$ generated by $\mathfrak{g}'$ is strictly contained in $\textbf{WC}(\widetilde{\E}\,, \widetilde{\B})^{\mathfrak{g}}$\,.

\label{RemarkStrictInclusion}

\end{rema}

\section{Reduction to the $\mathfrak{g}'$-harmonic (supersymmetric) tensors}

\label{SectionThree}

Let $\X'$ be the maximal $\B'$-isotropic subspace of $\V'$ defined in Section \ref{Section2}. We can then find a maximal $\B'$-isotropic subspace $\Y'$ of $\V'$ such that $\V' = \X' \oplus^{\perp} \Y'$, so $\Y' \cong \X'^{*}$. We denote by $\mathfrak{k}'$ and $\mathfrak{p}'$ the subsets of $\mathfrak{g}'$ given by
\begin{equation*}
\mathfrak{k}' = \left\{\T \in \mathfrak{g}'\,, \T(\X') \subseteq \X' \text{ and } \T(\Y') \subseteq \Y'\right\}\,, \quad \quad \mathfrak{p}' = \left\{\T \in \mathfrak{g}'\,, \T(\X') \subseteq \Y' \text{ and } \T(\Y') \subseteq \X'\right\}\,.
\end{equation*}
In particular, we have
\begin{equation*}
\mathfrak{k}' \cong \mathfrak{gl}(\X')\,, \qquad \left[\mathfrak{k}'\,, \mathfrak{k}'\right] \subseteq \mathfrak{k}'\,, \qquad \left[\mathfrak{k}'\,, \mathfrak{p}'\right] \subseteq \mathfrak{p}'\,, \qquad \left[\mathfrak{p}'\,, \mathfrak{p}'\right] \subseteq \mathfrak{k}'\,.
\end{equation*}
Moreover, we denote by $\mathfrak{n}'^{-}$ and $\mathfrak{n}'^{+}$ the subsets of $\mathfrak{p}'$ given by 
\begin{equation*}
\mathfrak{n}'^{-} = \left\{\T \in \mathfrak{g}'\,, \T(\X') \subseteq \left\{0\right\} \text{ and } \T(\Y') \subseteq \X'\right\}\,, \quad \quad \mathfrak{n}'^{+} = \left\{\T \in \mathfrak{g}'\,, \T(\X') \subseteq \Y' \text{ and } \T(\Y') \subseteq \left\{0\right\}\right\}\,.
\end{equation*}
One can easily see that 
\begin{equation*}
\left[\mathfrak{n}'^{-}\,, \mathfrak{n}'^{-}\right] = \left\{0\right\}\,, \qquad \left[\mathfrak{n}'^{-}\,, \mathfrak{n}'^{+}\right] \subseteq \mathfrak{k}'\,, \qquad \left[\mathfrak{n}'^{+}\,, \mathfrak{n}'^{+}\right] = \left\{0\right\}\,.
\end{equation*}

\noindent The Equation \eqref{DecompositionDuality} implies that 
\begin{equation*}
\S(\E)^{\mathfrak{n}'^{+}} = \left(\bigoplus\limits_{\pi \in \omega(\mathscr{G})} \V_{\pi} \otimes \V_{\theta(\pi)}\right)^{\mathfrak{n}'^{+}} = \bigoplus\limits_{\pi \in \omega(\mathscr{G})} \V_{\pi} \otimes \V^{\mathfrak{n}'^{+}}_{\theta(\pi)}\,.
\end{equation*}
Moreover, using that $\left[\mathfrak{k}'\,, \mathfrak{n}'^{+}\right] \subseteq \mathfrak{n}'^{+}$, it follows that $\V^{\mathfrak{n}'^{+}}_{\theta(\pi)}$ is a $\mathfrak{k}'$-module\,.

\begin{rema}

\begin{enumerate}
\item For all $\A \in \mathfrak{n}'^{+}$\,, $\B \in \mathfrak{n}'^{-}$\,, and $k \in \mathbb{Z}^{+}$, we have
\begin{equation}
\A \cdot \S^{k}(\E) \subseteq \S^{k-2}(\E)\,, \qquad \qquad \B \cdot \S^{k}(\E) \subseteq \S^{k+2}(\E)\,.
\label{DegreeN}
\end{equation}
(see Section \ref{SectionFour} for the case $\X' = \mathbb{C}^{1|1}$, the proof for the general case is similar)\,.
\item We denote by $\V(\pi\,, \mathfrak{n}'^{+})$ the subset of $\V(\pi)$ given by 
\begin{equation*}
\V(\pi\,, \mathfrak{n}'^{+}) = \V(\pi) \cap \S(\E)^{\mathfrak{n}'^{+}}\,.
\end{equation*}
One can see that $\V(\pi\,, \mathfrak{n}'^{+}) = \V^{\mathfrak{n}'^{+}}_{\theta(\pi)}$\,. Moreover, we have
\begin{equation*}
\S(\E)^{\mathfrak{n}'^{+}} = \bigoplus\limits_{d = 0}^{\infty} \S^{d}(\E)^{\mathfrak{n}'^{+}}\,,
\end{equation*}
where $\S^{d}(\E)^{\mathfrak{n}'^{+}} = \S^{d}(\E) \cap \S(\E)^{\mathfrak{n}'^{+}}$.
\item We denote by $\U(\mathfrak{n}'^{-})$\,, $\U(\mathfrak{k}')$\,, and $\U(\mathfrak{n}'^{+})$ the subalgebras of $\textbf{WC}(\widetilde{\E}\,, \widetilde{\B})$ generated by $\mathfrak{n}'^{-}$\,, $\mathfrak{k}'$\,, and $\mathfrak{n}'^{+}$ respectively. One can show that
\begin{equation*}
\U(\mathfrak{g}') := \textbf{WC}(\widetilde{\E}\,, \widetilde{\B})^{\mathscr{G}} = \U(\mathfrak{n}'^{-})\U(\mathfrak{k}')\U(\mathfrak{n}'^{+})\,.
\end{equation*}
\end{enumerate}

\end{rema}

\begin{lemma}

For all $\pi \in \omega(\mathscr{G})$, we get $\V^{\mathfrak{n}'^{+}}_{\theta(\pi)} \neq \left\{0\right\}$\,.

\label{LemmaVk}

\end{lemma}

\begin{proof}

Let $v = \sum\limits_{i = 1}^{n} v_i\in \V(\pi)$, with $v_{i} \in \S^{d_{i}}(\E)$, with $0 \leq d_{1} < d_{2} < \ldots < d_{n}$ and $v_n\not=0$. In particular, $\deg(v) = d_{n}$\,. 

\noindent If $n' \cdot v=0$ for all $n' \in \mathfrak{n}'^{+}$, then $v \in \V(\pi\,, \mathfrak{n}'^{+})$. However, if $\mathfrak{n}'^{+} \cdot v \neq \left\{0\right\}$, there exists $n_{1} \in \mathfrak{n}'^{+}$ such that $n_{1} \cdot v \neq 0$\,. 

\noindent Suppose that $\mathfrak{n}'^{+} \cdot v \neq \left\{0\right\}$ and let $v_{1} = n_{1} \cdot v \neq 0$. It follows from Equation \eqref{DegreeN} that the degree of $v_{1}$ is at most $d_{n} - 2$. Again, if $\mathfrak{n}'^{+} \cdot v_{1} = \left\{0\right\}$, then $v_{1} \in \V(\pi\,, \mathfrak{n}'^{+})$; otherwise, there exists $n_{2} \in \mathfrak{n}'^{-}$ such that $n_{2} \cdot v_{1} \neq 0$, with $\deg(n_{2} \cdot v_{1})$ is at most $d_{n} - 4$\,. 

\noindent By continuing this process, it follows that there exists $n = n_{k}n_{k-1} \ldots n_{2}n_{1} \in \U(\mathfrak{n}'^{-})$ such that $n \cdot v \neq 0$ and $\mathfrak{n}'^{-} (n\cdot v) = \left\{0\right\}$, i.e. the subspace $\V^{\mathfrak{n}'^{+}}_{\theta(\pi)}$ is non-zero\,.

\end{proof}

\begin{coro}

For all $\pi \in \omega(\mathscr{G})$, we have $\V_{\theta(\pi)} = \U(\mathfrak{n}'^{-})\V^{\mathfrak{n}'^{+}}_{\theta(\pi)}$\,. 

\label{CorollaryVk}

\end{coro}

\begin{proof}

First of all, we get from Lemma \ref{LemmaVk} that the subspace $\U(\mathfrak{n}'^{-})\V^{\mathfrak{n}'^{+}}_{\theta(\pi)}$ of $\V_{\theta(\pi)}$ is non-zero. Moreover, one can easily see that $\U(\mathfrak{n}'^{-})\V^{\mathfrak{n}'^{+}}_{\theta(\pi)}$ is $\mathfrak{g}'$-invariant, therefore the corollary follows by irreducibility of $\V_{\theta(\pi)}$ (see Theorem \ref{TheoremDualityAllanHadi})\,.

\end{proof}

\begin{rema}

For all $\pi \in \omega(\mathscr{G})$, the $\pi$-isotypic component $\V(\pi)$ is infinite dimensional. In particular, using that $\V_{\pi}$ is finite dimensional, it is equivalent to say that the representation $\V_{\theta(\pi)}$ of $\mathfrak{g}'$ is infinite dimensional\,.

\noindent However, the subspace $\V(\pi\,, \mathfrak{n}'^{+})$ is finite dimensional. Moreover, there exists a unique $d_{\pi} \in \mathbb{Z}^{+}$ such that $\V(\pi) \cap \S^{d_{\pi}}(\E) \neq \left\{0\right\}$ and $\V(\pi) \cap \S^{i}(\E) = \left\{0\right\}$ for all $ i < d_{\pi}$. It is straightforward to prove that $\V(\pi\,, \mathfrak{n}'^{+}) \subseteq \S^{d_{\pi}}(\E)$\,.

\end{rema}

\begin{prop}

For all $\pi \in \omega(\mathscr{G})$, the joint action of $\mathscr{G}$ and $\mathfrak{k}'$ on $\V(\pi\,, \eta'^{+})$ is irreducible. In other words, the $\mathfrak{k}'$-module $\V^{\mathfrak{n}'^{+}}_{\theta(\pi)}$ is irreducible\,.

\end{prop}

\begin{proof}

Let $\R$ be a proper $\mathscr{G} + \mathfrak{k}'$-invariant subspace of $\V(\pi\,, \mathfrak{n}'^{+})$. In particular, there exists a non-zero vector $v \in \V(\pi\,, \mathfrak{n}'^{+})$ such that $\mathbb{C}v \cap \R = \left\{0\right\}$. Let $\Z = \U(\mathfrak{n}'^{-})\R$. One can easily see that $\Z$ is a $\mathscr{G} + \mathfrak{g}'$-invariant subspace of $\V(\pi)$ and that $\Z$ does not contain $v$. However, $\V(\pi)$ is an irreducible $\mathscr{G} + \mathfrak{g}'$-module. Therefore, $\V(\pi\,, \mathfrak{n}'^{+})$ does not contain any proper $\mathscr{G} + \mathfrak{k}'$-invariant subspaces, and the proposition follows\,.

\end{proof}

\begin{rema}

Let $\pi_{1} \nsim\pi_{2}$ be two non-equivalent representations in $\omega(\mathscr{G})$. Then the corresponding irreducible $\mathfrak{k}'$-modules $\V^{\mathfrak{n}'^{+}}_{\theta(\pi_{1})}$ and $\V^{\mathfrak{n}'^{+}}_{\theta(\pi_{2})}$ are not isomorphic. Indeed, if $\V^{\mathfrak{n}'^{+}}_{\theta(\pi_{1})} \sim \V^{\mathfrak{n}'^{+}}_{\theta(\pi_{2})}$, then $\U(\mathfrak{n}'^{-})\V^{\mathfrak{n}'^{+}}_{\theta(\pi_{1})} \sim \U(\mathfrak{n}'^{-})\V^{\mathfrak{n}'^{+}}_{\theta(\pi_{2})}$ as $\mathfrak{g}'$-module, so it follows from Corollary \ref{CorollaryVk} that $\V_{\theta(\pi_{1})} \sim \V_{\theta(\pi_{2})}$, which contradicts Theorem \ref{TheoremDualityAllanHadi}\,.

\noindent Moreover, suppose that $d \geq 2$. One can prove that 
\begin{equation*}
\S^{d}(\E) = \S^{d}(\E)^{\mathfrak{n}'^{+}} \oplus \mathfrak{n}'^{-}\S^{d-2}(\E)\,.
\end{equation*}

\label{RemarkKAction}

\end{rema}

\noindent In particular, using that $\V^{\mathfrak{n}'^{+}}_{\theta(\pi)}$ is an highest weight module of $\mathfrak{k}'$, it follows from Corollary \ref{CorollaryVk} that $\V_{\theta(\pi)}$ is an highest weight module of $\mathfrak{g}'$\,.

\section{The Lie superalgebras $\mathfrak{g}$ and $\mathfrak{g}'$ as differential operators on $\S(\E)$}

\label{SectionFour}

From now on, we assume that $\V' = \mathbb{C}^{2|2}$. We start by fixing a basis for $\V$ and $\V'$\,. Let $\left\{u_{1}\,, u_{2}\,, \ldots\,, u_{2n}\right\}$ be a basis of the even part $\V_{\bar{0}} = \mathbb{C}^{2n|0}$ of $\V$ such that both spaces $\Span\left\{u_{1}\,, \ldots\,, u_{n}\right\}$ and $\Span\left\{u_{n+1}\,, \ldots\,, u_{2n}\right\}$ are $\B_{|_{\V_{\bar{0}}}}$-isotropic and such that
\begin{equation*}
\B(u_{i}\,, u_{2n+1-j}) = -\B(u_{2n+1-j}\,, u_{i}) = \delta_{i\,, j}\,, \qquad \left(1 \leq i\,, j \leq n\right)\,.
\end{equation*}
Moreover, let $\left\{u_{2n+1}\right\}$ be a basis of $\V_{\bar{1}} = \mathbb{C}^{0|1}$ such that $\B(u_{2n+1}\,, u_{2n+1}) = 1$\,. 

\begin{nota} 

For all $n \in \mathbb{Z}^{+}$, we denote by $\E_{n}$ the matrices in $\Mat(n \times n)$ with $1$ on the anti-diagonal and $0$ elsewhere, i.e.
\begin{equation*}
\textbf{\E}_{n} = \begin{pmatrix} 0 & 0 & \cdots & 0 & 1 \\ 0 & 0 & \cdots & 1 & 0 \\ \vdots & \vdots & \ddots & \vdots & \vdots \\ 0 & 1 & \cdots & 0 & 0 \\ 1 & 0 & \cdots & 0 & 0 \end{pmatrix}\,.
\end{equation*}
Moreover, for all $m\,, n \in \mathbb{Z}^{+}$, we denote by $\textbf{\O}_{m\,, n}$ the zero matrix in $\Mat(m \times n)$\,.

\end{nota}

\noindent In particular, $\mathscr{B} = \left\{u_{1}\,, u_{2}\,, \ldots\,, u_{2n}\,, u_{2n+1}\right\}$ is a basis of $\V$ and we have
\begin{equation*}
\Mat_{\mathscr{B}}(\B) = \begin{pmatrix} \textbf{\O}_{n\,, n} & \textbf{E}_{n} & \textbf{\O}_{n\,, 1} \\ -\textbf{E}_{n} & \textbf{\O}_{n\,, n} & \textbf{\O}_{n\,, 1} \\ \textbf{\O}_{1\,, n} & \textbf{\O}_{1\,, n} & 1 \end{pmatrix}\,.
\end{equation*}

\begin{rema}

Using the basis $\mathscr{B}$ of $\V = \mathbb{C}^{2n|1}$, the orthosymplectic Lie superalgebra $\mathfrak{spo}(2n|1)$ is given by
\begin{equation*}
\mathfrak{spo}(2n|1)_{\bar{0}} = \left\{\begin{pmatrix} \A & \B & \textbf{\O}_{n\,, 1} \\ \C & -\E_{n}\A^{t}\E_{n} & \textbf{\O}_{n\,, 1} \\ \textbf{\O}_{1\,, n} & \textbf{\O}_{1\,, n} & 0 \end{pmatrix}\,, \A\,, \B\,, \C \in \Mat(n \times n)\,, \B^{t} = \E_{n}\B\E_{n}\,, \C^{t} = \E_{n}\C\E_{n}\right\}\,,
\end{equation*}
\begin{equation*}
\mathfrak{spo}(2n|1)_{\bar{1}} = \left\{\begin{pmatrix} \textbf{\O}_{n\,, n} & \textbf{\O}_{n\,, n} & x \\ \textbf{\O}_{n\,, n} & \textbf{\O}_{n\,, n} & y \\ -y^{t}\E_{n} & x^{t}\E_{n} & 0 \end{pmatrix}\,, x\,, y \in \mathbb{C}^{n}\right\}\,.
\end{equation*}
Let $\mathfrak{t}_{\sp}$ be the subalgebra of $\mathfrak{sp}(2n)$ given by
\begin{equation*}
\mathfrak{t}_{\sp} = \bigoplus\limits_{i = 1}^{n} \mathbb{C}\left(\E_{i\,, i} - \E_{2n+1-i\,, 2n+1-i}\right)\,.
\end{equation*}
Using this previous realization of $\mathfrak{spo}(2n|1)$, we can take a Borel subalgebra $\mathfrak{b}_{\sp}$ of $\mathfrak{sp}(2n)$ of the form $\mathfrak{b}_{\sp} := \mathfrak{t}_{\sp} \oplus \mathfrak{n}^{+}_{\sp}$, where $\mathfrak{n}^{+}_{\sp}$ is made of upper triangular matrices. We denote by $\mathfrak{b}^{\spo}$ the Borel subalgebra of $\mathfrak{spo}(2n|1)$ whose even part $\mathfrak{b}^{\spo}_{\bar{0}}$ is $\mathfrak{b}_{\sp}$ and whose odd part $\mathfrak{b}^{\spo}_{\bar{1}}$ is given by
\begin{equation*}
\mathfrak{b}^{\spo}_{\bar{1}} = \bigoplus\limits_{k = 1}^{n} \mathbb{C}\left(\E_{k\,, 2n+1} + \E_{2n+1\,, 2n+1-k}\right)\,.
\end{equation*}

\label{BorelSpo2n1}

\end{rema}

\noindent Similarly, 
\begin{itemize}
\item Let $\left\{v_{1}\,, v_{2}\right\}$ be a basis of $\V'_{\bar{0}} = \mathbb{C}^{2|0}$ such that 
\begin{equation*}
\B'(v_{1}\,, v_{1}) = \B'(v_{2}\,, v_{2}) = 0\,, \qquad \qquad \B'(v_{1}\,, v_{2}) = \B'(v_{2}\,, v_{1}) = 1\,.
\end{equation*}
\item Let $\left\{w_{1}\,, w_{2}\right\}$ be a basis of $\V'_{\bar{1}} = \mathbb{C}^{0|2}$ such that 
\begin{equation*}
\B'(w_{1}\,, w_{1}) = \B'(w_{2}\,, w_{2}) = 0\,, \qquad \qquad \B'(w_{1}\,, w_{2}) = -\B'(w_{2}\,, w_{1}) = 1\,.
\end{equation*}
\end{itemize}

\noindent In particular, $\mathscr{B}' = \left\{v_{1}\,, v_{2}\,, w_{1}\,, w_{2}\right\}$ is a basis of $\V'$ such that 
\begin{equation*}
\Mat_{\mathscr{B}'}(\B') = \begin{pmatrix} 0 & 1 & 0 & 0 \\ 1 & 0 & 0 & 0 \\ 0 & 0 & 0 & 1 \\ 0 & 0 & -1 & 0 \end{pmatrix}\,.
\end{equation*}

\noindent Let $\X' = \X'_{\bar{0}} \oplus \X'_{\bar{1}} := \mathbb{C}v_{1} \oplus \mathbb{C}w_{1}$. For all $1 \leq i \leq 2n$, we denote by $x_{i}$ and $\eta_{i}$ the elements of $\E = \X' \otimes \V$ given by
\begin{equation*}
x_{i} = \begin{cases} u_{i} \otimes v_{1} & \text{ if } 1 \leq i \leq 2n \\ u_{2n+1} \otimes w_1 & \text{ if } i = 2n+1 \end{cases} \qquad \qquad \eta_{i} = \begin{cases} u_{i} \otimes w_{1} & \text{ if } 1 \leq i \leq 2n \\ u_{2n+1} \otimes v_{1} & \text{ if } i = 2n+1 \end{cases}
\end{equation*}

\noindent For all $1 \leq i \leq n$, we have
\begin{itemize}
\item $\widetilde{\B}(u_{i} \otimes v_{1}\,, u_{2n+1-i} \otimes v_{2}) = \B(u_{i}\,, u_{2n+1-i})\B'(v_{1}\,, v_{2}) = 1$
\item $\widetilde{\B}(u_{2n+1-i} \otimes v_{1}\,, u_{i} \otimes v_{2}) = \B(u_{2n+1-i}\,, u_{i})\B'(v_{1}\,, v_{2}) = -1$
\item $\widetilde{\B}(u_{2n+1} \otimes v_{1}\,, u_{2n+1} \otimes v_{2}) = \B(u_{2n+1}\,, u_{2n+1})\B'(v_{1}\,, v_{2}) = 1$
\item $\widetilde{\B}(u_{i} \otimes w_{1}\,, u_{2n+1-i} \otimes w_{2}) = \B(u_{i}\,, u_{2n+1-i})\B'(w_{1}\,, w_{2}) = 1$
\item $\widetilde{\B}(u_{2n+1-i} \otimes w_{1}\,, u_{i} \otimes w_{2}) = \B(u_{2n+1-i}\,, u_{i})\B'(w_{1}\,, w_{2}) = -1$
\item $\widetilde{\B}(u_{2n+1} \otimes w_{1}\,, u_{2n+1} \otimes w_{2}) = -\B(u_{2n+1}\,, u_{2n+1})\B'(w_{1}\,, w_{2}) = -1$
\end{itemize}
In particular, the dual basis is given by
\begin{equation*}
x^{*}_{i} = \begin{cases} u_{2n+1-i} \otimes v_{2} & \text{ if } 1 \leq i \leq n \\ -u_{2n+1-i} \otimes v_{2} & \text{ if } n+1 \leq i \leq 2n \\ -u_{2n+1} \otimes w_{2} & \text{ if } i = 2n+1 \end{cases} \qquad \qquad \eta^{*}_{i} = \begin{cases} u_{2n+1-i} \otimes w_{2} & \text{ if } 1 \leq i \leq n \\ -u_{2n+1-i} \otimes w_{2} & \text{ if } n+1 \leq i \leq 2n \\ u_{2n+1} \otimes v_{2} & \text{ if } i = 2n+1 \end{cases}
\end{equation*}

\begin{nota}

For all $1 \leq i \leq 2n + 1$, we denote by $x_{i}\,, \eta_{i}\,, \partial_{x_{i}}$ and $\partial_{\eta_{i}}$ the elements of $\textbf{WC}(\widetilde{\E}\,, \widetilde{\B})$ given by
\begin{equation*}
x_{i} := \M_{x_{i}} = \iota(x_{i})\,, \qquad \quad \eta_{i} := \M_{\eta_{i}} = \iota(\eta_{i})\,, \qquad \quad \partial_{x_{i}} := \D_{x^{*}_{i}} = \iota(x^{*}_{i})\,, \qquad \quad \partial_{\eta_{i}} = \iota(\eta^{*}_{i})\,,
\end{equation*}
where $\iota$ is the map given in Equation \eqref{MapIota}\,.
\end{nota}

We now give a realization of $\mathfrak{g}'$ as differential operators on $\S(\E)$. We keep the notation of Section \ref{Section2}. We get 
\begin{equation*}
\Omega^{\mathfrak{g}} = \mathfrak{g}'\,, \qquad \qquad \sigma(\Omega^{\mathfrak{g}}) = \S^{2}(\widetilde{\E})^{\mathfrak{g}}\,,
\end{equation*}
\noindent where $\Omega$ is defined in Equation \eqref{SuperalgebraOmega}, and $\sigma: \textbf{WC}(\widetilde{\E}\,, \widetilde{\B}) \to \S(\widetilde{\E})$ is the Weyl symbol (see \cite[Equation~5.3]{CHENGWANG1}). As explained in \cite{SERGEEV}, we have that $\left(\V \otimes \V\right)^{\mathfrak{g}} = \mathbb{C}\cdot\theta$ where
\begin{equation*}
\theta = \sum\limits_{i = 1}^{2n} u_{i} \otimes u^{*}_{i} + u_{2n+1} \otimes u^{*}_{2n+1} = \sum\limits_{i = 1}^{n}\left(u_{i} \otimes u_{2n+1-i} - u_{2n+1-i} \otimes u_{i}\right) + u_{2n+1} \otimes u_{2n+1}\,.
\end{equation*}

\begin{nota}

For an element $\tau = a_{1} \otimes b_{1} \otimes a_{2} \otimes b_{2} \in \V \otimes \V' \otimes \V \otimes \V' \in \left(\V \otimes \V'\right)^{\otimes 2} = \widetilde{\E}^{\otimes 2}$, we denote by $\gamma(\tau)$ the corresponding element of $\S^{2}(\V \otimes \V')$ given by
\begin{equation*}
\gamma(\tau) = \cfrac{a_{1} \otimes b_{1} \otimes a_{2} \otimes b_{2} + (-1)^{\left|a_{1} \otimes b_{1}\right| \cdot \left|a_{2} \otimes b_{2}\right|} a_{2} \otimes b_{2} \otimes a_{1} \otimes b_{1}}{2}\,.
\end{equation*}

\end{nota}

\noindent Using that 
\begin{equation*}
\S^{2}(\widetilde{\E}) = \left(\widetilde{\E} \otimes \widetilde{\E}\right)^{\mathscr{S}_{2}}\,,
\end{equation*}
where the action of the symmetric group $\mathscr{S}_{2}$ on $\widetilde{\E}^{\otimes 2}$ is given by 
\begin{equation*}
\left(1 2\right) \cdot \tilde{e}_{1} \otimes \tilde{e}_{2} = (-1)^{\left|\tilde{e}_{1}\right| \cdot \left|\tilde{e}_{2}\right|} \tilde{e}_{2} \otimes \tilde{e}_{1}\,,
\end{equation*}
we obtain that

\medskip

{\small

\begin{eqnarray*}
&&\S^{2}(\widetilde{\E})^{\mathfrak{g}} = \left(\left(\widetilde{\E} \otimes \widetilde{\E}\right)^{\mathscr{S}_{2}}\right)^{\mathfrak{g}} = \left(\left((\V \otimes \V')^{\otimes 2}\right)^{\mathscr{S}_{2}}\right)^{\mathfrak{g}} = \left(\left(\V^{\otimes 2} \otimes \V'^{\otimes 2}\right)^{\Delta(\mathscr{S}_{2} \times \mathscr{S}_{2})}\right)^{\mathfrak{g}} = \left(\left(\V^{\otimes 2}\right)^{\mathfrak{g}} \otimes \V'^{\otimes 2}\right)^{\Delta(\mathscr{S}_{2} \times \mathscr{S}_{2})} \\
   & = & \left\{\sum\limits_{i = 1}^{n}\left(u_{i} \otimes u_{2n+1-i} - u_{2n+1-i} \otimes u_{i}\right)\otimes \alpha \otimes \beta + u_{2n+1} \otimes u_{2n+1} \otimes \alpha \otimes \beta, \alpha\,, \beta \in \V'\right\}^{\Delta(\mathscr{S}_{2} \times \mathscr{S}_{2})} \\ 
      & = & \left\{\sum\limits_{i = 1}^{n}\left(u_{i} \otimes  \alpha \otimes u_{2n+1-i} \otimes \beta - u_{2n+1-i} \otimes \alpha \otimes u_{i} \otimes \beta\right) + (-1)^{\alpha} u_{2n+1} \otimes \alpha \otimes u_{2n+1} \otimes \beta\,, \alpha\,, \beta \in \V'\right\}^{\mathscr{S}_{2}} \\
       & = & \left\{\sum\limits_{i = 1}^{n}\left(\gamma(u_{i} \otimes  \alpha \otimes u_{2n+1-i} \otimes \beta) - \gamma(u_{2n+1-i} \otimes \alpha \otimes u_{i} \otimes \beta)\right) + (-1)^{\alpha} \gamma(u_{2n+1} \otimes \alpha \otimes u_{2n+1} \otimes \beta)\,, \alpha\,, \beta \in \V\right\}
\end{eqnarray*}}

\begin{rema}

\noindent For $\alpha$ and $\beta$, we take elements in the basis $\left\{v_{1}\,, v_{2}\,, w_{1}\,, w_{2}\right\}$ of $\V'$, 
and we obtain the following differential operators
\begin{itemize}
\item $\F_{1\,, 1} = \sum\limits_{i = 1}^{n}\left(x_{i}\partial_{x_{i}} + x_{2n+1-i}\partial_{x_{2n+1-i}}\right) + \eta_{2n+1}\partial_{\eta_{2n+1}} + \frac{2n-1}{2}$
\item $\F_{2\,, 2} = \sum\limits_{i = 1}^{n}\left(\eta_{i}\partial_{\eta_{i}} + \eta_{2n+1-i}\partial_{\eta_{2n+1-i}}\right) + x_{2n+1}\partial_{x_{2n+1}} - \frac{2n-1}{2}$
\item $\F_{2\,, 1}  = x_{2n+1}\partial_{\eta_{2n+1}} - \sum\limits_{i = 1}^{n}\left(\eta_{i}\partial_{x_{i}} + \eta_{2n+1-i}\partial_{x_{2n+1-i}}\right)$
\item $\F_{1\,, 2} = \eta_{2n+1}\partial_{x_{2n+1}} - \sum\limits_{i = 1}^{n}\left(x_{i}\partial_{\eta_{i}} + x_{2n+1-i}\partial_{\eta_{2n+1-i}}\right)$
\end{itemize}
Here $2n-1$ is the superdimension of $\mathbb{C}^{2n|1}$.
We have $\mathfrak{k}' = \Span\left\{\F_{1\,, 1}\,, \F_{2\,, 2}\,, \F_{1\,, 2}\,, \F_{2\,, 1}\right\}\simeq\mathfrak{gl}(1|1)$, 
with the Cartan subalgebra $\mathfrak{t}' = \Span\left\{\F_{1\,, 1}\,, \F_{2\,, 2}\right\}$\,. 

\noindent Similarly, we obtain the operators
\begin{itemize}
\item $\D_{12} = \partial_{x_{2n+1}}\partial_{\eta_{2n+1}} + \sum\limits_{i = 1}^{n}\left(\partial_{x_{i}}\partial_{\eta_{2n+1-i}} - \partial_{x_{2n+1-i}}\partial_{\eta_{i}}\right)$
\item $\D_{22} = \partial^{2}_{x_{2n+1}} + 2\sum\limits_{i = 1}^{n}\partial_{\eta_{i}}\partial_{\eta_{2n+1-i}}$
\item $\R_{12} = x_{2n+1}\eta_{2n+1} + \sum\limits_{i = 1}^{n}\left(\eta_{i}x_{2n+1-i} - \eta_{2n+1-i}x_{i}\right)$
\item $\R_{22} = x^{2}_{2n+1} + 2\sum\limits_{i = 1}^{n}\eta_{i}\eta_{2n+1-i}$
\end{itemize}
In particular, we get $\mathfrak{n}'^{+} = \Span\left\{\D_{12}\,, \D_{22}\right\}$, $\mathfrak{n}'^{-} = \Span\left\{\R_{12}\,, \R_{12}\right\}$, and 
$\mathfrak{g}'=\mathfrak{n}'^{-}\oplus\mathfrak{k}'\oplus\mathfrak{n}'^{+}\simeq\mathfrak{osp}(2|2).$

\label{RemarkNotationF}

\end{rema}

\begin{rema}

In this paper, we will not need the explicit operators for the action of $\mathfrak{g}=\mathfrak{spo}(2n+1|1)$. However, the techniques used for $\mathfrak{g}'$ will also work for $\mathfrak{g}$. Indeed, using \cite{SERGEEV}, we get that $\left(\V' \otimes \V'\right)^{\mathfrak{g}'} = \mathbb{C}\cdot\theta'$ where
\begin{equation*}
\theta' = \sum\limits_{i = 1}^{2} v_{i} \otimes v^{*}_{i} + \sum\limits_{j = 1}^{2} w_{j} \otimes w^{*}_{j} = \left(v_{1} \otimes v_{2} + v_{2} \otimes v_{1}\right) + \left(w_{1} \otimes w_{2} - w_{2} \otimes w_{1}\right)\,,
\end{equation*}
and then it follows that the set $\S^{2}(\widetilde{\E})^{\mathfrak{g}'}$ is given by
{\small
\begin{equation*}
\left\{\gamma(\alpha \otimes v_{1} \otimes \beta \otimes v_{2}) + \gamma(\alpha \otimes v_{2} \otimes \beta \otimes v_{1}) + (-1)^{\left|\beta\right|} \big(\gamma(\alpha \otimes w_{1} \otimes \beta \otimes w_{2}) - \gamma(\alpha \otimes w_{2} \otimes \beta \otimes w_{1})\big)\,, \alpha\,, \beta \in \V\right\}
\end{equation*}}

\end{rema}

\section{Howe duality for the pair $\left(\mathfrak{h}\,, \mathfrak{h}'\right) = \left(\mathfrak{gl}(2n|1)\,, \mathfrak{gl}(1|1)\right)$}

\label{SectionGLGL}

In this section, we recall some results of \cite{CHENGWANG2} for the pair $\left(\mathfrak{h}\,, \mathfrak{h}'\right) := \left(\mathfrak{gl}(2n|1)\,, \mathfrak{gl}(1|1)\right)$. Let $\E = \mathbb{C}^{2n|1} \otimes \mathbb{C}^{1|1}$. Using the natural actions of $\mathfrak{h}$ and $\mathfrak{h}'$ on $\E$, we get that both $\mathfrak{h}$ and $\mathfrak{h}'$ are subalgebras of $\mathfrak{gl}(\E)$. Moreover, one can see that 
\begin{equation*}
\mathscr{C}_{\mathfrak{gl}(\E)}\left(\mathfrak{h}\right) = \mathfrak{h}'\,, \qquad \qquad \mathscr{C}_{\mathfrak{gl}(\E)}\left(\mathfrak{h}'\right) = \mathfrak{h}\,,
\end{equation*}
i.e. $\left(\mathfrak{h}\,, \mathfrak{h}'\right)$ is a dual pair in $\mathfrak{gl}(\E)$. Moreover, as explained in \cite{CHENGWANG2}, we have that $\left(\mathfrak{h}\,, \mathfrak{h}'\right)$ is an irreducible reductive dual pair in $\mathfrak{spo}(\widetilde{\E}\,, \widetilde{\B})$ (of type II)\,. 

\noindent We now look at the joint action of $\mathfrak{h}$ and $\mathfrak{h}'$ on $\S(\E)$. Note that both $\mathfrak{h}$ and $\mathfrak{h}'$ act on $\S(\E)$ by finite dimensional modules\,.

\begin{rema}

As in Section \ref{SectionFour}, we denote by $x_{1}\,, x_{2}\,, \ldots\,, x_{2n+1}$ and $\eta_{1}\,, \eta_{2}\,, \ldots\,, \eta_{2n+1}$ the basis of $\E$, with $\E_{\bar{0}} = \Span\left\{x_{1}\,, \ldots\,, x_{2n+1}\right\}$ and $\E_{\bar{1}} = \Span\left\{\eta_{1}\,, \ldots\,, \eta_{2n+1}\right\}$. Again, the embedding of $\mathfrak{h}$ and $\mathfrak{h}'$ is given by differential operators. We denote by $\E_{i\,, j}$ the standard basis of $\mathfrak{gl}(2n|1)$ and by $\epsilon_{i\,, j}$ the corresponding differential operators in $\textbf{WC}(\widetilde{\E}\,, \widetilde{\B})$.  For all $1 \leq i\,, j \leq 2n$, we have 
\begin{align*}
& \epsilon_{i\,, j} = x_{i}\partial_{x_{j}} + \eta_{i}\partial_{\eta_{j}} \qquad & \epsilon_{2n+1\,, 2n+1} = x_{2n+1}\partial_{x_{2n+1}} + \eta_{2n+1}\partial_{\eta_{2n+1}} \\ 
& \epsilon_{i\,, 2n+1} = \eta_{i}\partial_{x_{2n+1}} + x_{i}\partial_{\eta_{2n+1}} \qquad & \epsilon_{2n+1\,, j} = \eta_{2n+1}\partial_{x_{j}} + x_{2n+1}\partial_{\eta_{j}}
\end{align*}
Similarly, we denote by $\F_{k\,, l}$ the standard basis of $\mathfrak{gl}(1|1)$ and by $\zeta_{k\,, l}$ the corresponding differential operators in $\textbf{WC}(\widetilde{\E}\,, \widetilde{\B})$. We get
\begin{align*}
& \zeta_{1\,, 1} = \sum\limits_{t = 1}^{2n}x_{t}\partial_{x_{t}} + \eta_{2n+1}\partial_{\eta_{2n+1}} \qquad & \zeta_{2\,, 2} = \sum\limits_{t = 1}^{2n} \eta_{k}\partial_{\eta_{k}} + x_{2n+1}\partial_{x_{2n+1}} \\ 
& \zeta_{1\,, 2} = \sum\limits_{t = 1}^{2n} x_{t}\partial_{\eta_{t}} - \eta_{2n+1}\partial_{x_{2n+1}} \qquad & \zeta_{2\,, 1} = \sum\limits_{t = 1}^{2n} \eta_{t}\partial_{x_{t}} - x_{2n+1}\partial_{\eta_{2n+1}} 
\end{align*}

\label{EmbeddingCW}

\end{rema}

\begin{rema}

One can see that the operators giving the actions of $\mathfrak{gl}(1|1)$ for the two dual pairs are identical to the one obtained in Section \ref{SectionFour} (up to a shift by half of the superdimension of $\mathbb{C}^{2n|1}$ on the Cartan subalgebra of $\mathfrak{k}'$)\,.

\label{ShiftOnCartan}

\end{rema}

\noindent We denote by $\mathfrak{t} = \Span\left\{\epsilon_{i\,, i}\,, 1 \leq i \leq 2n+1\right\}$ and $\mathfrak{t}' = \Span\left\{\zeta_{1\,, 1}\,, \zeta_{2\,,2}\right\}$ the Cartan subalgebras of $\mathfrak{h}$ and $\mathfrak{h}'$ and by $\mathfrak{b}$ and $\mathfrak{b}'$ the standard Borel subalgebras of $\mathfrak{gl}(2n|1)$ and $\mathfrak{gl}(1|1)$ respectively consisting of upper triangular matrices\,.

\noindent As explained in \cite{CHENGWANG2}, the Howe duality for $\left(\mathfrak{h}\,, \mathfrak{h}'\right)$ follows from the Schur-Sergeev duality for the general linear superalgebra. More particularly, we have

\begin{theo}

For all $d \in \mathbb{Z}^{+}$, we have
\begin{equation}
\S^{d}(\E) = \bigoplus\limits_{\lambda \in \Omega^{d}} \V^{2n|1}_{\sigma(\lambda)} \otimes \V^{1|1}_{\gamma(\lambda)}\,,
\label{DecompositionSd}
\end{equation}
where $\Omega^{d}$ corresponds to the set of Young diagrams $\lambda$ of size $d$ such that $\lambda_{2} \leq 1$, and where $\V^{2n|1}_{\sigma(\lambda)}$ (resp. $\V^{1|1}_{\gamma(\lambda)}$) is an irreducible $\mathfrak{b}$-highest weight module of $\mathfrak{h}$ (resp. an irreducible $\mathfrak{b}'$-highest weight module of $\mathfrak{h}'$) of $\mathfrak{b}$-highest weight $\sigma(\lambda)$ (resp. $\mathfrak{b}'$-highest weight $\gamma(\lambda)$) given by 
\begin{equation*}
\sigma(\lambda) = \left(\lambda_{1}\,, \lambda_{2}\,, \ldots\,, \lambda_{2n}\,, \langle\lambda'_{1} - 2n\rangle\right)\,, \qquad \qquad \left(\text{resp. }\gamma(\lambda) = \left(\lambda_{1}\,, \langle\lambda'_{1} - 1\rangle\right)\right)\,,
\end{equation*}
with $\langle \cdot \rangle: \mathbb{Z} \to \mathbb{Z}^{+}$ given by $\langle p \rangle = \max(0\,, p)$\,.

\label{TheoremGLGL}

\end{theo}

\begin{proof}

See \cite[Theorem~3.2]{CHENGWANG2}\,.

\end{proof}

\begin{nota}

For all $a \in \mathbb{Z}^{+}$, we denote by $0_{a}$ and $1_{a}$ the vectors of $\mathbb{Z}^{a}$ given by 
\begin{equation*}
0_{a} = \left(0\,, 0\,, \ldots\,, 0\right)\,, \qquad \qquad 1_{a} = \left(1\,, 1\,, \ldots\,, 1\right)\,.
\end{equation*}

\end{nota}

\begin{rema}

In particular, all the diagrams in $\Omega^{d}$ consists of Young diagrams with $d$ boxes, a unique row with $k$ boxes, with $1 \leq k \leq d$, and a unique column with $d-k+1$ boxes, i.e.
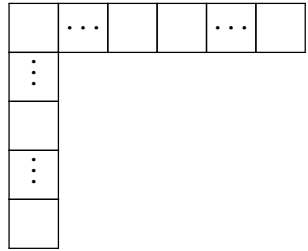
\begin{equation*}
\begin{ytableau}
~ & \cdots & ~ & ~ & \cdots & ~ \\
\vdots & \none & \none & \none & \none & \none \\
~ & \none & \none & \none & \none & \none \\
\vdots & \none & \none & \none & \none & \none \\
~ & \none & \none & \none & \none & \none
\end{ytableau}
\end{equation*}

\begin{nota}

We say that a vector $v \in \S(\E)$ is a $\mathfrak{b} + \mathfrak{b}'$-joint highest weight vector for $\mathfrak{h}+\mathfrak{h}'$ of highest weight $\lambda \times \lambda'$ (with $\lambda \in \mathfrak{t}^{*}$ and $\lambda' \in \mathfrak{t}'^{*}$) if for all $\H \in \mathfrak{t}$ and $\H' \in \mathfrak{t}'$, we get
\begin{equation*}
\mathfrak{b} \cdot v = \left\{0\right\}\,, \qquad \quad  \mathfrak{b}' \cdot v = \left\{0\right\}\,, \qquad \quad  \H v = \lambda(\H)v\,, \qquad\quad \H' v = \lambda'(\H')v\,.
\end{equation*}

\end{nota}

\end{rema}

\noindent Using \cite{CHENGWANG2}, we can obtain explicit formulas for the  joint highest weight vectors in Equation \eqref{DecompositionSd}
(with highest weights given with respect to the Borel subalgebras $\mathfrak{b}$ and $\mathfrak{b}'$).
For a~proof, see Lemma \ref{l_FormulasHWV} in Apppendix \ref{FormulasHWV}.

\begin{theo}

Let $d \geq 0$. For $k = 1\,, \ldots\,, 2n$, put 
\begin{equation*}
\omega_{d\,, k} := x^{d}_{1} \sum\limits_{i = 1}^{k}(-1)^{i-1}x_{i} \prod\limits_{\underset{a \neq i}{a = 1}}^{k}\eta_{a}\,.
\end{equation*}
Similarly, for $\ell>0$, let
\begin{equation*}
{\omega}_{d\,, 2n+\ell} := x^{d}_{1}x^{\ell}_{2n+1}\sum\limits_{i = 1}^{2n}(-1)^{i-1}x_{i} \prod\limits_{\underset{a \neq i}{a = 1}}^{2n}\eta_{a} 
- \ell x^{d}_{1}x^{\ell-1}_{2n+1}\prod\limits_{a=1}^{2n+1}\eta_{a}\,.
\end{equation*}
Then $\omega_{d\,, k}$ is a~joint $\mathfrak{b}+\mathfrak{b}'$-highest weight vector in $\S^{d+k}(\E)$ of highest weights
\begin{equation*}
\left(d+1\,, 1_{m-1}\,, 0_{2n-m}\,, k-m \right) \qquad \text{ and } \qquad \left(d+1\,,k-1\right)
\end{equation*}
with $m = \min(k\,, 2n)$. In addition, the joint highest weight vector $\omega_{0\,, 0} = 1$ has weights $\left(0_{2n+1}\right)$ and $\left(0\,, 0\right)$\,.

\label{TheoremHighestWeight1}

\end{theo}

\begin{rema}

The orthosymplectic Lie superalgebra $\mathfrak{spo}(2n|1)$ is a subalgebra of $\mathfrak{gl}(2n|1)$. We denote by $\mathfrak{b}^{\spo}$ the Borel subalgebra of $\mathfrak{spo}(2n|1)$ defined in Remark \ref{BorelSpo2n1}. One can see that $\mathfrak{b}^{\spo}$ is not included in $\mathfrak{b}$. Therefore, a $\mathfrak{b}$-highest weight vector is not, in general, a $\mathfrak{b}^{\spo}$-highest weight vector.  

\noindent We denote by $\widetilde{\mathfrak{b}}$ the Borel subalgebra of $\mathfrak{gl}(2n|1)$ obtained by replacing the odd root vector $\E_{i\,, 2n + 1}$ of $\mathfrak{b}$ by $\E_{2n+1\,, i}$ for all $n+1 \leq i \leq 2n$. In particular, we have
\begin{equation*}
\widetilde{\mathfrak{b}}_{\bar{0}} = \mathfrak{b}_{\bar{0}}\,, \qquad \qquad   \widetilde{\mathfrak{b}}_{\bar{1}} = \bigoplus\limits_{i = 1}^{n} \mathbb{C}\E_{i\,, 2n+1} \oplus \bigoplus\limits_{i = n+1}^{2n} \mathbb{C}\E_{2n+1\,, i}\,.
\end{equation*}

\noindent One can see that then $\mathfrak{b}^{\spo} \subseteq \widetilde{\mathfrak{b}}$. The goal now is to transform the $\mathfrak{b}$-highest modules obtained in Theorem \ref{TheoremHighestWeight1} into $\widetilde{\mathfrak{b}}$-highest weight modules\,.

\label{tildeBorel}

\end{rema}

\noindent For the roots of $\mathfrak{gl}(2n|1)$, we use the notations of \cite{CHENGWANG1}. We denote by $\left\{\delta_{1}\,, \ldots\,, \delta_{2n}\,, \varepsilon_{1}\right\}$ the basis of $\mathfrak{t}^{*}$ such that
\begin{equation*}
\delta_{i}(\E_{j\,, j}) = \begin{cases} 1 & \text{ if } j = i \\ 0 & \text{ otherwise } \end{cases} \qquad \varepsilon_{1}(\E_{j\,, j}) = \begin{cases} 1 & \text{ if } j = 2n+1 \\ 0 & \text{ otherwise } \end{cases} \qquad \qquad \left(1 \leq i \leq 2n\right)\,.
\end{equation*}
For all $1 \leq i \neq j \leq 2n$, we have
\begin{equation*}
\mathfrak{gl}(2n|1)_{\delta_{i} - \delta_{j}} = \mathbb{C}\E_{i\,, j}\,.
\end{equation*}
Moreover, for all $1 \leq i \leq 2n$, we have
\begin{equation*}
\mathfrak{gl}(2n|1)_{\delta_{i} - \varepsilon_{1}} = \mathbb{C}\E_{i\,, 2n+1}\,, \qquad \qquad \mathfrak{gl}(2n|1)_{\varepsilon_{1} - \delta_{i}} = \mathbb{C}\E_{2n+1\,, i}\,.
\end{equation*}
Let $\left(\cdot\,, \cdot\right)$ be the map defined on $\mathfrak{gl}(2n|1)$ by 
\begin{equation*}
\left(\A\,, \B\right) = \str(\A\B)\,,
\end{equation*}
where $\str$ is the supertrace on $\mathfrak{gl}(2n|1)$ (see \cite[Section~1.1.2]{CHENGWANG1}). The form $\left(\cdot\,, \cdot\right)$ is non-degenerate and its restriction to $\mathfrak{t}$ is also non-degenerate. We still denote by $\left(\cdot\,, \cdot\right)$ the corresponding non-degenerate form on $\mathfrak{t}^{*}$. As explained in \cite[Equation~1.21]{CHENGWANG1}, for all $1 \leq i\,, j\leq 2n$, we have
\begin{equation*}
\left(\delta_{i}\,, \delta_{j}\right) = \begin{cases} 1 & \text{ if } i = j \\ 0 & \text{ otherwise } \end{cases} \qquad \qquad \left(\delta_{i}\,, \varepsilon_{1}\right) = 0 \qquad \qquad  \left(\varepsilon_{1}\,, \varepsilon_{1}\right) = -1\,.
\end{equation*}
In particular, for all $1 \leq i \leq 2n$, the roots $\delta_{i} - \varepsilon_{1}$ are isotropic. 

\begin{nota}

We denote by $\mathfrak{b}^{1}$ the Borel subalgebra obtained by replacing in $\mathfrak{b}$ the root vector $\E_{2n\,, 2n+1}$ by $\E_{2n+1\,, 2n}$ (the root vector corresponding to $\varepsilon_{1} - \delta_{2n}$). Similarly, for all $2 \leq i \leq n$, we denote by $\mathfrak{b}^{i}$ the Borel subalgebra of $\mathfrak{g}$ obtained by replacing in $\mathfrak{b}^{i-1}$ the root vectors $\E_{2n+1-i\,, 2n+1}$ by $\E_{2n+1\,, 2n+1-i}$. 

\noindent By convention, we have $\mathfrak{b}^{0} := \mathfrak{b}$\,.

\end{nota}

\begin{prop}

Let $1 \leq i \leq n$, and let $\left(\pi\,, \V_{\pi}\right)$ be an irreducible finite dimensional representation of $\mathfrak{gl}(2n|1)$ of $\mathfrak{b}^{i-1}$-highest weight vector $v$ and of $\mathfrak{b}^{i-1}$-highest weight $\lambda$\,.
\begin{enumerate}
\item If $\lambda\left(\E_{2n+1-i\,, 2n+1-i} + \E_{2n+1\,, 2n+1}\right) = 0$, then $\V_{\pi}$ is a $\mathfrak{gl}(2n|1)$-module of $\mathfrak{b}^{i}$-highest weight vector $v$ and $\mathfrak{b}^{i}$-highest weight $\lambda$\,.
\item If $\lambda\left(\E_{2n+1-i\,, 2n+1-i} + \E_{2n+1\,, 2n+1}\right) \neq 0$, then $\V_{\pi}$ is a $\mathfrak{gl}(2n|1)$-module of $\mathfrak{b}^{i}$-highest weight vector $\E_{2n+1\,, 2n+1-i}v$ and $\mathfrak{b}^{i}$-highest weight $\lambda - (\delta_{2n+1-i}-\varepsilon_{1})$\,.
\end{enumerate}

\label{PropositionChengWang}

\end{prop}

\begin{proof}

Using that $\delta_{2n+1-i} -  \varepsilon_{1}$ is a simple odd root for $\mathfrak{b}^{i-1}$, the proposition follows from \cite[Lemma~1.40]{CHENGWANG1}\,.

\end{proof}

\noindent Using Theorem \ref{TheoremHighestWeight1} and Proposition \ref{PropositionChengWang}, we get the following theorem\,.

\begin{theo}

Let $d \geq 0$. For $k = 1\,, \ldots\,, n$, put $\widetilde{\omega}_{d\,, k} := \omega_{d\,, k}$ where $\omega_{d\,, k}$ is defined in Theorem \ref{TheoremHighestWeight1}.
For $\ell > 0$, denote
\begin{equation*}
\widetilde{\omega}_{d\,, n+\ell} := x^{d}_{1}x^{\ell}_{2n+1}\sum\limits_{i = 1}^{n}(-1)^{i-1}x_{i} \prod\limits_{\underset{a \neq i}{a = 1}}^{n}\eta_{a} 
- \ell x^{d}_{1}x^{\ell-1}_{2n+1}\eta_{2n+1}\prod\limits_{a=1}^{n}\eta_{a}
\end{equation*}
Then $\widetilde{\omega}_{d\,, k}$ is a joint $\widetilde{\mathfrak{b}}+\mathfrak{b}'$-highest weight vector in $\S^{d+k}(\E)$ of highest weights respectively given by $\widetilde{\lambda}_{d\,, k}$ and $\left(d+1\,, k-1\right)$, where  for all $1 \leq k \leq n$ and $l \in \mathbb{N}^{*}$, we have
\begin{equation*}
\widetilde{\lambda}_{d\,, k} = \left(d+1\,, 1_{k-1}\,, 0_{2n+1-k} \right) \qquad \text{ and } \qquad \widetilde{\lambda}_{d\,, n+\ell} = \left(d+1\,, 1_{n-1}\,, 0_{n}\,, \ell \right)\,.
\end{equation*}
Finally, the joint highest weight vector $\widetilde{\omega}_{0\,, 0} = 1$ has weights $\left(0_{2n+1}\right)$ and $\left(0\,, 0\right)$\,.

\label{TheoremWeightRightBorel}

\end{theo}




\begin{proof}

This is a direct consequence of Proposition \ref{PropositionChengWang}. Indeed, we only need to apply the operators $\prod\limits_{k = 1}^{n} \E_{2n+1\,, n+k}$ to the vectors $\omega_{d\,, k}$ given in Theorem \ref{TheoremHighestWeight1}. A different proof can be found in Appendix \ref{FormulasHWV}\,.

\end{proof}

\begin{rema}

\begin{enumerate}
\item Using that $\mathfrak{h}$ and $\mathfrak{h}'$ supercommutes in $\mathfrak{spo}(\widetilde{\E}\,, \widetilde{\B})$, it follows that the $\widetilde{\mathfrak{b}}$-highest weight vectors obtained in Theorem \ref{TheoremWeightRightBorel} are $\mathfrak{b}'$-highest weight vectors, whose weights are given in Theorem \ref{TheoremHighestWeight1}\,.
\item Theorem \ref{TheoremWeightRightBorel} is going to be very useful in the next section. We have $\mathfrak{b}^{\spo} \subseteq \widetilde{\mathfrak{b}}$. In particular, if $\left(\pi\,, \V_{\pi}\right)$ is a $\widetilde{\mathfrak{b}}$-highest weight module of $\widetilde{\mathfrak{b}}$-highest vector $v$ and highest weight $\lambda = \left(\lambda_{1}\,, \ldots\,, \lambda_{2n+1}\right)$, then $\U = \mathscr{U}(\mathfrak{spo}(2n|1)) \cdot v \subseteq \V_{\pi}$ is a $\mathfrak{b}^{\spo}$-highest weight module of $\mathfrak{spo}(2n|1)$ of $\mathfrak{b}^{\spo}$-highest weight vector $v$ and highest weight $\tau \in \mathbb{Z}^{n}$ given by
\begin{equation*}
\tau = \left(\lambda_{1} - \lambda_{2n}\,, \lambda_{2} - \lambda_{2n-1}\,, \ldots\,, \lambda_{n} - \lambda_{n+1}\right)\,.
\end{equation*}
\end{enumerate}

\label{RemarkRestriction}

\end{rema}

\section{Howe duality for the pair $\left(\mathfrak{spo}(2n|1)\,, \mathfrak{osp}(2|2)\right)$}

Using the results obtained in Section \ref{SectionGLGL}, we obtained the following theorem\,.

\begin{theo}

We keep the notations of Theorem \ref{TheoremWeightRightBorel}\,. 
Then the vector $\widetilde{\omega}_{d\,, k}\in \S^{d+k}(\E)$ is a $\mathfrak{b}^{\spo}$-highest weight vector of $\mathfrak{spo}(2n|1)$ of highest weight $\tau_{d\,, k}$\,, where $\tau_{d\,, k}$ is given by 
\begin{equation*}
\tau_{d\,, k} : = \begin{cases} \left(d+1\,, 1_{k-1}\,, 0_{n-k} \right) & \text{ if } 1 \leq k \leq n \\  \left(d+1\,, 1_{n-1} \right) & \text{ otherwise }  \end{cases}
\end{equation*}
Moreover, $\tau_{0\,, 0} = \left(0_{n}\right)$\,.

\label{TheoremPart1}

\end{theo}

\begin{proof}

The theorem follows from Theorem \ref{TheoremWeightRightBorel} and Remark \ref{RemarkRestriction}\,.

\end{proof}

\noindent As mentioned in Section \ref{SectionThree}, the duality for the pair $\left(\mathfrak{spo}(2n|1)\,, \mathfrak{osp}(2|2)\right)$ acting on $\S(\E)$ can be obtained by looking at the joint action of $\left(\mathfrak{spo}(2n|1)\,, \mathfrak{gl}(1|1)\right)$ on $\S(\E)^{\mathfrak{n}^{+}}$. The following proposition tells us which one of the highest weight vectors given in Theorem \ref{TheoremPart1} are in $\S(\E)^{\mathfrak{n}^{+}}$.

\begin{prop}

Among all the highest weight vectors given in Theorem \ref{TheoremPart1}, the only one that belongs to $\S(\E)^{\mathfrak{n}'^{+}}$ are the vectors 
$\widetilde{\omega}_{d\,, k}$ for $d \geq 0$ and $k \leq n+1$.

\label{PropositionVectorHarmonics}

\end{prop}

\begin{proof}
Let $k\leq n$.
As explained in Section \ref{SectionFour}, we have $\mathfrak{n}'^{+} = \Span\left\{\D_{12}\,, \D_{22}\right\}$, with 
\begin{equation*}
\D_{12} = \partial_{x_{2n+1}}\partial_{\eta_{2n+1}} + \sum\limits_{i = 1}^{n}\left(\partial_{x_{i}}\partial_{\eta_{2n+1-i}} - \partial_{x_{2n+1-i}}\partial_{\eta_{i}}\right)\,, \qquad \qquad \D_{22} = \partial^{2}_{x_{2n+1}} + 2\sum\limits_{i = 1}^{n}\partial_{\eta_{i}}\partial_{\eta_{2n+1-i}}\,.
\end{equation*}
First of all, one can see that the vector $\widetilde{\omega}_{d\,, k} = \omega_{d\,, k}$ only depends on $x_{1}\,, \ldots\,, x_{n}\,, \eta_{1}\,, \ldots\,, \eta_{n}$. In particular, we have $\D_{12}\widetilde{\omega}_{d,k} = \D_{22}\widetilde{\omega}_{d,k} = 0$\,.

\noindent Now suppose that $k = n+1$. We have
\begin{equation*}
\widetilde{\omega}_{d\,, n+1} = x^{d}_{1}x_{2n+1}\sum\limits_{j = 1}^{n}(-1)^{j-1}x_{j}\prod\limits_{\underset{a \neq j}{a=1}}^{n}\eta_{a} 
- x^{d}_1\eta_{2n+1}\prod\limits_{k = 1}^{n}\eta_{a}\,,
\end{equation*}
and  $\D_{12}\widetilde{\omega}_{d\,, n+1} = \D_{22}\widetilde{\omega}_{d\,, n+1} = 0$\,. 

\noindent Let $\ell > 1$. Then we have
\begin{equation*}
\D_{12}\widetilde{\omega}_{d\,, n+\ell} = -\ell(\ell-1)x^{d}_{1}x^{\ell-2}_{2n+1}\prod\limits_{k = 1}^{n}\eta_{a}\,,
\end{equation*}
i.e. $\D_{12}\widetilde{\omega}_{d\,, n+\ell} \neq 0$. Therefore, $\widetilde{\omega}_{d,n+\ell} \notin \S(\E)^{\mathfrak{n}'^{+}}$\,.
\end{proof}

\begin{coro}

As a $\mathfrak{spo}(2n|1) + \mathfrak{gl}(1|1)$-module, we get
\begin{equation*}
\V^{2n|1}_{\left(0_{n}\right)} \otimes \V^{1|1}_{\left(\alpha\,, -\alpha\right)} \oplus 
\bigoplus\limits_{d = 0}^{\infty} \bigoplus\limits_{k = 1}^{n} 
\V^{2n|1}_{\left(d+1,1_{k-1},0_{n-k}\right)} \otimes \V^{1|1}_{\left(d+1+\alpha,k-1-\alpha\right)} 
\oplus \bigoplus\limits_{d = 0}^{\infty} 
\V^{2n|1}_{\left(d+1,1_{n-1}\right)} \otimes \V^{1|1}_{\left(d+1+\alpha,n-\alpha\right)} \subseteq \S(\E)^{\mathfrak{n}'^{+}}
\end{equation*}
where $\alpha=\tfrac{2n-1}{2}$\,.
\label{LastCorollary}

\end{coro}


\begin{rema}

As mentioned in the introduction even for $n=1$, the inclusion in Corollary \ref{LastCorollary} is proper. 
Thus we do not obtain all joint highest weight vectors in $\S(\E)^{\mathfrak{n}'^{+}}$ from Howe duality $(\mathfrak{gl}(2n|1),\mathfrak{gl}(1|1))$. 
This is in contrast with the classical case. Let us explain how we can generate the additional joint highest weight vectors from one variable cases.  Indeed, as $\mathfrak{spo}(2n|1)$-modules, we have
\begin{equation*}
\S(\E)\cong \S(\V)\otimes\Lambda(\V)
\end{equation*}
using the coordinates $x_{1}\,, \ldots, x_{2n}\,,\eta_{2n+1}$ and $\eta_{1}\,, \ldots\,, \eta_{2n}\,, x_{2n+1}$ in the first and second copy of $\V=\mathbb{C}^{2n|1}$, respectively. It is easy to show that all highest weight vectors in $\S(\V)$ are $x^{d}_{1}$ for $d \geq 0$.
Using Howe duality $(\textbf{SpO}(2n|1),\mathfrak{sp}(2))$, all harmonic highest weight vectors in $\Lambda(\V)$ can be described as in \cite[Section 5.3]{COULEMBIER}, 
see Proposition \ref{LastProposition} below. Here an element $\p \in \Lambda(\V)$ is harmonic if and only if $\D_{22}\p=0$. 
The laplacian $\D_{22}$ is defined as in Remark \ref{RemarkNotationF}.

\label{OneVariable}
\end{rema}

\begin{nota}

Let $\I \subseteq \left\{1\,, \ldots\,, 2n\right\}$ be an index set and $\I = \left\{i_{1}\,, \ldots\,, i_{a}\right\}$ with $1 \leq i_{1} < \ldots < i_{a} \leq 2n$. Then denote $\left|\I\right| = a$, $\eta(\I) = \eta_{i_{1}} \cdots \eta_{i_{a}}$ and 
$\underline{\I} = \left\{\underline{i_{a}}\,, \underline{i_{a-1}}\,, \ldots\,, \underline{i_{1}}\right\}$ with $\underline{i} = 2n+1-i$. Finally, for any integer $0 \leq k\leq 2n$, we define $\left[k\right ] = \left\{1\,, \ldots\,, k\right\}$\,.

\label{NotaI}

\end{nota}

\begin{prop}

For all $k=n+1,\ldots, 2n+1$, define the vector $\Delta_{k} \in \Lambda^{k}(\V)$ as
\begin{equation*}
\Delta_{k} = \sum\limits_{a = 0}^{k-n-1} b_{a} x^{2(k-n-a)-1}_{2n+1}
\left(\sum\limits_{\I \subseteq \left\{2n-k+2\,, \ldots\,, n\right\}\,, \left|\I\right| = a} \eta(\left[2n+1-k\right]\,, \I\,, \underline{\I})\right)
\end{equation*}
where  $b_{0} = 1$, $b_{a} = \prod\limits_{r = 0}^{a-1}(2(k-n-r)-1)$ and 
$\eta\left(\left[2n+1-k\right]\,, \I\,, \underline{\I}\right) = \eta(\left[2n+1-k\right] \cup \I \cup \underline{\I})$. Then the complete list of harmonic highest weight vectors in $\Lambda(\V)$ is given by
\begin{equation*}
\s_{n\,, k} = \begin{cases} 
\eta(\left[k\right]) & \text{ if } 0 \leq k \leq n, \\ 
\Delta_{k} & \text{ if } n+1 \leq k \leq 2n+1.
\end{cases} 
\end{equation*}
In addition, $\s_{n\,, k}\in\Lambda^{k}(\V)$ has $\mathfrak{spo}(2n|1)$-highest weight $\left(1_{m}\,, 0_{n-m}\right)$ with $m = \min\left(k\,,2n+1-k\right)$\,.

\label{LastProposition}

\end{prop}

\noindent Before we move to the explicit decomposition of $\S(\E)^{\mathfrak{n}'^{+}}$ and $\S(\E)$, we prove the following lemma\,.

\begin{lemma}

The highest weights of the $\mathfrak{spo}(2n|1)$-modules appearing in $\S(\E)$ are of the form
\begin{equation}
\left(a\,, 1_{k}\,, 0_{n-k-1}\right)
\label{PossibleWeights}
\end{equation}
with $a \in \mathbb{N}^{*}$ and $0 \leq k \leq n-1$\,.

\label{LastLemma}
    
\end{lemma}

\begin{proof}

Let $\left(\pi\,, \V\right)$ be an irreducible (finite dimensional) representation of $\mathfrak{spo}(2n|1)$ of highest weight $\lambda$ such that $\Hom_{\mathfrak{spo}(2n|1)}(\V\,, \S(\E)) \neq \left\{0\right\}$. The restriction of $\pi$ to $\mathfrak{sp}(2n)$ is not irreducible but the $\mathfrak{sp}(2n)$-module of highest weight $\lambda$ appears as a submodule in $\V$. Therefore, it is enough to prove that the only highest weights for $\mathfrak{sp}(2n)$ appearing in $\S(\E)$ are the ones given in Equation \eqref{PossibleWeights}.

\noindent As a $\mathfrak{sp}(2n)$-module, we have $\E = \mathbb{C}^{2n|2n} \oplus \mathbb{C}^{1|1}$, with $\mathfrak{sp}(2n)$ acting trivially on $\mathbb{C}^{1|1}$. In particular, we have that
\begin{equation}
\S(\E) = \S(\mathbb{C}^{2n|2n}) \otimes \S(\mathbb{C}^{1|1})\,.
\label{DecompositionProofLemma}
\end{equation}
As explained in \cite[Theorem~5.31]{CHENGWANG1}, the only highest weights appearing in the decomposition of $\mathfrak{sp}(2n) \curvearrowright \S(\mathbb{C}^{2n|2n})$ are of the form $\left(a\,, 1_{k}\,, 0_{n-k-1}\right)$. Using that $\mathfrak{sp}(2n)$ is acting trivially on $\mathbb{C}^{1|1}$, it follows from Equation \eqref{DecompositionProofLemma} that $\lambda$ is of the form $\left(a\,, 1_{k}\,, 0_{n-k-1}\right)$, and the lemma follows\,.

\end{proof}








\noindent We can now state our main theorems.

\begin{theo} As a $\mathfrak{spo}(2n|1) + \mathfrak{gl}(1|1)$-module, we have
\begin{align*}
\S(\E)^{\mathfrak{n}'^{+}} &= \V^{2n|1}_{\left(0_{n}\right)} \otimes \V^{1|1}_{\left(\alpha\,, -\alpha\right)} \oplus
\V^{2n|1}_{\left(0_{n}\right)} \otimes \V^{1|1}_{\left(1+\alpha\,, 2n-\alpha\right)}\\ 
 &\oplus \bigoplus\limits_{d = 0}^{\infty} \bigoplus\limits_{k = 1}^{n}\left( 
\V^{2n|1}_{\left(d+1,1_{k-1},0_{n-k}\right)} \otimes \V^{1|1}_{\left(d+1+\alpha,k-1-\alpha\right)} \oplus
\V^{2n|1}_{\left(d+1,1_{k-1},0_{n-k}\right)} \otimes \V^{1|1}_{\left(d+1+\alpha,2n-k-\alpha\right)}\right)\,,
\end{align*}
where $\alpha=\tfrac{2n-1}{2}$, $\V^{2n|1}_{\lambda}$ is the irreducible $\mathfrak{spo}(2n|1)$-module of highest weight $\lambda$, and 
$\V^{1|1}_{\left(a\,, b\right)}$ is the irreducible $\mathfrak{gl}(1|1)$-module of highest weight $\left(a\,, b\right)$\,.
For explicit formulas for the joint highest weight vectors in $\S(\E)^{\mathfrak{n}'^{+}}$, see Lemma \ref{l_FormulasHHWV} in Appendix \ref{FormulasHWV}.
\label{HarmonicDecomposition}

\end{theo}

\begin{proof}
We use the decomposition $\S(\E) \cong \S(\V)\otimes\Lambda(\V)$ described in Remark \ref{OneVariable}. Recall that, for the action of $\mathfrak{spo}(2n|1)$, $x^{d}_{1}$ are highest weight vectors in $\S(\V)$ and $\s_{n\,, k}$ for $k=0\,, \ldots\,, 2n+1$ are all harmonic highest weight vectors in $\Lambda(\V)$, see Proposition \ref{LastProposition}. For $d \geq 0$ and $k = 0\,, \ldots\,, 2n+1$, denote $\q_{d\,, k} = x^{d}_{1}\s_{n\,, k}$. Then $\q_{d\,, k} \in \S^{d+k}(\E)$ is a~highest weight vector for the action of $\mathfrak{spo}(2n|1)$. It is easy to see that $\D_{12}\q_{d\,,k} \neq 0$ if and only if $d > 0$ and $k = 2n+1$. Moreover, for $k > 0$, $\p_{d\,, k} = \F_{1\,, 2}\q_{d\,, k} \neq 0$. Note also that $\p_{d\,, 1} = -x_1^{d+1}$. 
For explicit formulas for $\p_{d\,, k}$, see Lemma \ref{l_FormulasHHWV} in Appendix \ref{FormulasHWV}.
Actually, we show that $\p_{0\,, 0} = 1$, $\p_{0\,, 2n+1}$ and $\p_{d\,, k}$ for $d \geq 0$, $k = 1\,, \ldots\,, 2n$ are all joint highest weight vectors in $\S(\E)^{\mathfrak{n}'^{+}}$. For a given $d \geq 0$ and $k = 1\,, \ldots\,, n$, harmonic joint highest weight vectors $\p_{d\,, k} \in \S^{d+k}(\E)$ and $\p_{d\,,2n+1-k} \in \S^{d+2n+1-k}(\E)$ have the same $\mathfrak{spo}(2n|1)$-weight $\left(d+1\,, 1_{k-1}\,, 0_{n-k}\right)$ and 
$\mathfrak{gl}(1|1)$-weights $\left(d+1+\alpha\,, k-1-\alpha\right)$ and $\left(d+1+\alpha\,, 2n-k-\alpha\right)$ respectively. They generate irreducible submodules 
\begin{equation}
\V^{2n|1}_{\left(d+1,1_{k-1}\,, 0_{n-k}\right)} \otimes \V^{1|1}_{\left(d+1+\alpha\,, k-1-\alpha\right)} \qquad \text{ and } \qquad \V^{2n|1}_{\left(d+1\,, 1_{k-1},0_{n-k}\right)} \otimes \V^{1|1}_{\left(d+1+\alpha\,, 2n-k-\alpha\right)}
\label{EquationLastProof}
\end{equation}
contained in $\S^{d+k}(\E)^{\mathfrak{n}'^{+}}$ and $\S^{d+2n+1-k}(\E)^{\mathfrak{n}'^{+}}$ respectively. Note that the two $\mathfrak{spo}(2n|1)$-modules obtained in Equation \eqref{EquationLastProof} are not isomorphic as $\textbf{SpO}(2n|1)$-modules (see Remark \ref{LastRemark}). We get a similar result for $\p_{0\,, 0} = 1$ and $\p_{0\,, 2n+1}$. Using Lemma \ref{LastLemma}, it is clear that in this way we get all irreducible submodules in an isotypic decomposition of $\S(\E)^{\mathfrak{n}'^{+}}$ 
under the joint action of $\textbf{SpO}(2n|1)$ and $\mathfrak{gl}(1|1)$, see Section \ref{SectionThree}\,.  

\end{proof}

\begin{theo} 

As a $\mathfrak{spo}(2n|1) + \mathfrak{osp}(2|2)$-module, we have
\begin{align*}
\S(\E) &= \V^{2n|1}_{\left(0_{n}\right)} \otimes \V^{2|2}_{\left(\alpha\,, -\alpha\right)} \oplus
\V^{2n|1}_{\left(0_{n}\right)} \otimes \V^{2|2}_{\left(1+\alpha\,, 2n-\alpha\right)} \oplus\\ 
 &\oplus \bigoplus\limits_{d = 0}^{\infty} \bigoplus\limits_{k = 1}^{n}\left( 
\V^{2n|1}_{\left(d+1,1_{k-1},0_{n-k}\right)} \otimes \V^{2|2}_{\left(d+1+\alpha,k-1-\alpha\right)} \oplus
\V^{2n|1}_{\left(d+1,1_{k-1},0_{n-k}\right)} \otimes \V^{2|2}_{\left(d+1+\alpha,2n-k-\alpha\right)}\right)\,,
\end{align*}
where $\alpha=\tfrac{2n-1}{2}$, $\V^{2n|1}_{\lambda}$ is the irreducible $\mathfrak{spo}(2n|1)$-module of highest weight $\lambda$, and 
$\V^{2|2}_{\left(a\,, b\right)}$ is the irreducible $\mathfrak{osp}(2|2)$-module of highest weight $\left(a\,, b\right)$\,.

\label{MainTheorem}

\end{theo}

\noindent We finish with a few remarks\,.

\begin{rema}
    
\begin{enumerate}
\item In our paper, we gave a decomposition of $\S(\E)$ for the irreducible reductive dual pair $\left(\mathfrak{spo}(2n|1)\,, \mathfrak{osp}(2|2)\right)$. All the joint highest weight vectors and highest weights are explicitly given (see Theorem \ref{MainTheorem} and Lemma \ref{l_FormulasHHWV}). As explained in Theorem \ref{TheoremDualityAllanHadi}, we still have a one-to-one correspondence of irreducible modules for the dual pair $\left(\textbf{SpO}(2n|1)\,, \mathfrak{osp}(2r|2s)\right)$\,.

\noindent Using the see-saw pair
\begin{equation*}
\begin{tikzcd}
\mathfrak{gl}(2n|1) \arrow[d, dash] \arrow[rd, dash]  & \mathfrak{osp}(2|2) \\
\mathfrak{spo}(2n|1) \arrow[ru, dash] & \arrow[u, dash] \mathfrak{gl}(1|1)
\end{tikzcd}
\end{equation*}
and the explicit description of the joint highest weights for the pair $\left(\mathfrak{gl}(2n|1)\,, \mathfrak{gl}(1|1)\right)$ obtained in \cite{CHENGWANG2}, we were able to generate highest weight vectors of some of the representations appearing in the duality for $\left(\mathfrak{spo}(2n|1)\,, \mathfrak{osp}(2|2)\right)$, but not all of them. In particular, using the following see-saw pair
\begin{equation*}
\begin{tikzcd}
\mathfrak{gl}(2n|1) \arrow[d, dash] \arrow[rd, dash]  & \mathfrak{osp}(2r|2s) \\
\mathfrak{spo}(2n|1) \arrow[ru, dash] & \arrow[u, dash] \mathfrak{gl}(r|s)
\end{tikzcd}
\end{equation*}
it is also possible to get the highest weight vector of some modules appearing in the decomposition of $\S(\mathbb{C}^{2n|1} \otimes \mathbb{C}^{r|s})$. 
However, we could not obtain an explicit and understandable description of the joint highest weight vectors using the results of \cite{CHENGWANG1}\,.
\item As mentioned in the proof of Theorem \ref{HarmonicDecomposition}, for all $\lambda = \left(u\,, 1_{k}\,, 0_{n-k-1}\right)$, with $u \in \mathbb{N}^{*}$ and $0 \leq k \leq n-1$, 
there exists an even number $d_{1}$ and an odd number $d_{2}$ such that $\Hom_{\mathfrak{g}}(\V^{2n|1}_{\lambda}\,, \S^{d_{1}}(\E)) \neq \left\{0\right\}$ and 
$\Hom_{\mathfrak{g}}(\V^{2n|1}_{\lambda}\,, \S^{d_{2}}(\E)) \neq \left\{0\right\}$. However, using the same argument as in Remark \ref{LastRemarkIntroduction}, 
one can show that these two irreducible representations are not isomorphic as $\textbf{SpO}(2n|1)$-modules. Therefore, the fact that the same $\lambda$'s appear in both $\S^{+}(\E)$ and $\S^{-}(\E)$ does not contradict the results of Section \ref{Section2}, i.e. a multiplicity free decomposition of $\S(\E)$ under the joint action of $\textbf{SpO}(2n|1)$ and $\mathfrak{osp}(2|2)$\,.
\item One can see that all the representations of $\mathfrak{gl}(1|1)$ appearing in $\S^{d}(\E)^{\mathfrak{n}'^{+}}\,, d \geq 1\,,$ are typical (see \cite[Definition~2.29]{CHENGWANG2}). In particular, all the irreducible $\mathfrak{gl}(1|1)$-modules are $2$-dimensional (see \cite{KAC})\,.
\item In general, we cannot expect a decomposition of $\S(\E)$ as in Equation \eqref{DecompositionDuality}. Indeed, in general, the action of $\mathfrak{g}$ on $\S(\E)$ is not semisimple. However, it does not mean that a one-to-one correspondence between some irreducible representations of $\mathfrak{g}$ and $\mathfrak{g}'$ cannot be obtained. As proved by Roger Howe in \cite{HOWETRANSCENDING} for the symplectic case (in the real case), submodules are replaced by representations that can be realized as a quotient of $\S(\E)$. As far as we know, such duality has not been obtained yet for Lie superalgebras (in the complex and real case). The paper of Howe and Lu (in the real case, see \cite{HOWELU}) is the first paper in this direction for the dual pair $\left(\O(p\,, q)\,, \widetilde{\textbf{SpO}}(2|2)\right)$, but this is beyond the scope of this paper\,.
\end{enumerate}

\label{LastRemark}

\end{rema}

\appendix

\section{Explicit formulas for highest weight vectors} 

\label{FormulasHWV}

First we prove Theorems \ref{TheoremHighestWeight1} and \ref{TheoremWeightRightBorel} describing joint highest weight vectors in $\S(\E)$ for the joint action of $\mathfrak{gl}(2n|1)$ and $\mathfrak{gl}(1|1)$. 
Let us consider the standard Borel subalgebras $\mathfrak{b}$ and $\mathfrak{b}'$ of $\mathfrak{gl}(2n|1)$ and $\mathfrak{gl}(1|1)$, respectively.
As $\mathfrak{gl}(2n|1)$-modules, we can identify 
\begin{equation*}
\S(\E) \cong \S(\V)\otimes\Lambda(\V)
\end{equation*}
using the coordinates $x_{1}\,, \ldots\,, x_{2n}\,, \eta_{2n+1}$ 
and $\eta_{1}\,, \ldots\,, \eta_{2n}\,, x_{2n+1}$ in the first and second copy of $\V = \mathbb{C}^{2n|1}$ respectively. 
It is easy to show that all $\mathfrak{b}$-highest weight vectors in $\S(\V)$ are $x^{d}_{1}$ for $d \geq 0$, and 
in $\Lambda(\V)$ we have $\nu_{n\,, k} = \eta(\left[k\right])$ for $k = 0\,, \ldots\,, 2n$ and $\nu_{n\,, 2n+\ell} = \eta(\left[2n\right])x^{\ell}_{2n+1}$ for $\ell > 0$ (we use here the Notations of \ref{NotaI})\,.

\begin{lemma}

\begin{enumerate}
\item  For all $d \geq 0$ and $k > 0$, the element $\omega_{d\,, k}$ defined as
\begin{equation*}
\omega_{d\,, k} = \zeta_{1\,, 2}\left(x_1^d \nu_{n\,,k}\right)
\end{equation*} 
is a~joint $\mathfrak{b}+\mathfrak{b}'$-highest weight vector in $\S^{d+k}(\E)$ of highest weights
$\left(d+1\,, 1_{m-1}\,,k-m \right)$ and $\left(d+1\,,k-1\right)$ with $m = \min(k\,, 2n)$. In addition, for $d \geq 0$, the joint highest weight vector $\omega_{d\,, 0} = x^{d}_{1}$ has (joint) weights $\left(d\,, 0_{2n}\right)$ and $\left(d\,, 0\right)$. Note that $\omega_{d\,, 1} = x^{d+1}_{1} = \omega_{d+1\,, 0}$ for $d \geq 0$\,.
\item Let $d \geq 0$. For $k = 1\,, \ldots\,, 2n$, we have that $\omega_{d\,, k} = x^{d}_{1}\Gamma(\left[k\right])$ where,
for $\I = \left\{i_{1}\,, \ldots\,, i_{a}\right\}$ with $1 \leq  i_{1} < \ldots < i_{a} \leq 2n$, we denote 
\begin{equation*}
\Gamma(\I) = \sum\limits_{j = 1}^{a}(-1)^{j-1}x_{i_{j}} \eta\left(\I \setminus\left\{i_{j}\right\}\right)\,.
\end{equation*}
For $\ell > 0$, we have that 
\begin{equation*}
\omega_{d\,, 2n + \ell} = x^{d}_{1}\left(x^{\ell}_{2n+1}\Gamma(\left[2n\right]) - \ell x^{\ell-1}_{2n+1}\eta\left(\left[2n+1\right]\right)\right)\,.
\end{equation*}
\item Let $d \geq 0$. For $k = 1\,, \ldots\,, n$, denote $\widetilde{\omega}_{d\,, k} = \omega_{d\,, k}$, and for $\ell>0$ put
\begin{equation*}
\widetilde{\omega}_{d\,, n + \ell} = x^{d}_{1}\left(x^{\ell}_{2n+1}\Gamma\left(\left[n\right]\right) - \ell x^{\ell-1}_{2n+1}\eta_{2n+1}\eta\left(\left[n\right]\right) \right)\,.
\end{equation*}
Then $\widetilde{\omega}_{d\,, k}$ is a~joint $\widetilde{\mathfrak{b}}+\mathfrak{b}'$-highest weight vector in $\S^{d+k}(\E)$ of highest weights $\widetilde{\lambda}_{d\,, k}$ and $\left(d+1\,, k-1\right)$. In addition, we have $\widetilde{\lambda}_{d\,, k} = \left(d+1\,, 1_{k-1}\,, 0_{2n+1-k} \right)$ for $k = 1\,, \ldots\,, n$, and
$\widetilde{\lambda}_{d\,, n+\ell} = \left(d+1\,, 1_{n-1}\,, 0_{n}\,, \ell \right)$ for $\ell > 0$\,.
Here the Borel subalgebra $\widetilde{\mathfrak{b}}$ is given in Remark \ref{tildeBorel}\,.
\end{enumerate}

\label{l_FormulasHWV}

\end{lemma}

\begin{proof} 

We start with the proof of $\left(2\right)$. For $\I = \left\{i_{1}\,, \ldots\,, i_{a}\right\}$ with $1 \leq i_{1} < \ldots < i_{a} \leq 2n$, we have that $\Gamma(\I)$ is equal to $\frac{1}{(a-1)!}\det(\M)$ , where $\M$ is a matrix of size $a \times a$ defined as
\begin{equation*}
\M = \begin{pmatrix}
x_{i_{1}} & \eta_{i_{1}} & \cdots & \eta_{i_{1}}\\
x_{i_{2}} & \eta_{i_{2}} & \cdots & \eta_{i_{2}}\\
\vdots & \vdots & \ddots & \vdots\\
x_{i_{a}} & \eta_{i_{a}} & \cdots & \eta_{i_{a}}
\end{pmatrix}
\end{equation*}
and $\det(\M) = \sum_{\sigma \in \mathscr{S}_{a}} \varepsilon(\sigma)\M_{\sigma(1)1}\M_{\sigma(2)2}\cdots \M_{\sigma(a)a}$ is the column determinant of $\M$ (here $\mathscr{S}_{a}$ is the symmetric group and $\varepsilon(\sigma)$ is the signature of $\sigma \in \mathscr{S}_{a}$, see \cite[Section 5.2.3]{CHENGWANG1})\,.
Recall that 
\begin{equation*}
\zeta_{1\,, 2} = \sum\limits_{t = 1}^{2n} x_{t}\partial_{\eta_{t}} - \eta_{2n+1}\partial_{x_{2n+1}}
\end{equation*}
For $k = 1\,, \ldots\,, 2n$, we have $\zeta_{1\,, 2}\left(\eta(\left[k\right])\right) = \Gamma(\left[k\right])$, which implies the formulas in $\left(2\right)$. Then the statement of $\left(1\right)$ is obvious\,.

\noindent We finish with the proof of $\left(3\right)$. Let $d \geq0$\,.
\begin{enumerate}
\item Assume that $\ell = 1\,, \ldots\,, n$. Using Proposition \ref{PropositionChengWang}, it is clear that 
\begin{equation*}
\widetilde{\omega}_{d\,, n + \ell} = \left(\E_{2n+1\,,n+1} \cdots \E_{2n+1\,, n + \ell}\right)\omega_{d\,, n+\ell}\,.
\end{equation*}
Recall that $\E_{2n+1\,, j} = \eta_{2n+1}\partial_{x_{j}} + x_{2n+1}\partial_{\eta_{j}}$. To get the formula for $\widetilde{\omega}_{d\,, n+\ell}$ (possibly up to a sign) notice that
$\partial_{x_{j}}\Gamma(\left[j\right]) = (-1)^{j-1}\eta(\left[j-1\right])$ and $\partial_{\eta_j}\Gamma(\left[j\right]) = (-1)^{j}\Gamma(\left[j-1\right])$ for $j = 1\,, \ldots\,, 2n$\,. 
\item Assume that $\ell > 0$. Again, it follows from Proposition \ref{PropositionChengWang} that
\begin{equation*}
\widetilde{\omega}_{d\,, 2n + \ell} = \left(\E_{2n+1\,,n+1} \cdots \E_{2n+1\,, 2n}\right)\omega_{d\,,2n+\ell}\,. 
\end{equation*}
Then one can continue as in (1)\,.
\end{enumerate}

\end{proof}

Now let us consider Howe duality $\left(\textbf{SpO}(2n|1)\,, \mathfrak{osp}(2|2)\right)$ on $\text{S}(\E)$. In the following lemma, we derive explicit formulas for the harmonic joint highest weight vectors 
$\p_{d\,, k}$ in $\text{S}(\E)$ defined in the proof of Theorem \ref{HarmonicDecomposition} (possibly up to a sign)\,.

\begin{lemma} 

For $k = 1\,, \ldots\,, n$, we have $\p_{d\,, k} = x^{d}_{1}\Gamma(\left[k\right])$ and 
\begin{equation*}
\p_{d\,, n+1} = x^{d}_{1}\left(x_{2n+1}\Gamma(\left[n\right]) -\eta_{2n+1}\eta(\left[n\right]) \right)\,.
\end{equation*}
For $k = n+2\,,\ldots\,, 2n+1$, we have 
\begin{align*}
\p_{d\,, k} &= x^{d}_{1}\sum\limits_{a = 0}^{k-n-1} b_{a}x^{2(k-n-a)-1}_{2n+1}
\left(\sum\limits_{\I \subseteq \{2n-k+2\,, \ldots\,, n\}\,, \left|\I\right| = a} \Gamma(\left[2n+1-k\right]\,, \I\,, \underline{\I})\right) \\
 & -x^{d}_{1}\sum\limits_{a = 0}^{k-n-1} b_{a+1}x^{2(k-n-a)-2}_{2n+1}\eta_{2n+1}
\left(\sum\limits_{\I \subseteq \{2n-k+2\,, \ldots\,, n\}\,, \left|\I\right| = a} \eta(\left[2n+1-k\right]\,, \I\,, \underline{\I})\right)
\end{align*}
where $b_{0} = 1$ and $b_{a} = \prod\limits_{r = 0}^{a-1}(2(k-n-r)-1)$\,. 

\label{l_FormulasHHWV}

\end{lemma}

\begin{proof}
For $k = n+1\,, \ldots\,, 2n+1$, we have 
$\p_{d\,, k} = x^{d}_{1}\zeta_{1\,, 2}\left(\Delta_{k}\right)$ and 
\begin{equation*}
\Delta_{k} = \sum\limits_{a = 0}^{k-n-1} b_{a} x^{2(k-n-a)-1}_{2n+1}
\left(\sum\limits_{\I \subseteq \left\{2n-k+2\,, \ldots\,, n\right\}\,, \left|\I\right| = a} \eta(\left[2n+1-k\right]\,, \I\,, \underline{\I})\right)
\end{equation*}
Using that $\zeta_{1\,, 2}\left(\eta(\left[k\right])\right) = \Gamma(\left[k\right])$, the proof is done. The other cases are obvious\,.

\end{proof}



\end{document}